\documentclass[11pt, leqno]{amsart}

\usepackage{shuffle} 
\usepackage{ytableau}
\usepackage{adjustbox}
\usepackage{tikz-cd}
\usepackage{tikz}
\usepackage{adjustbox}

\usepackage{lscape}
\usepackage{graphicx}
\usepackage{amssymb,amsmath,amsthm,amscd}
\usepackage{mathrsfs}
\usepackage{enumerate}
\usepackage[pdfstartview=FitV,
 linkcolor=blue,citecolor=blue,urlcolor=blue]{hyperref}

\usepackage[all]{xy}
\usepackage[normalem]{ulem}  

\usepackage{enumitem}

\allowdisplaybreaks[2]
\usepackage{oldgerm}
\usepackage{xspace}

\usepackage{marginnote}
\usepackage{float}

\newcommand{\nc}{\newcommand}
\numberwithin{equation}{section}

\newenvironment{red}{\relax\color{red}}{\relax}
\newenvironment{blue}{\relax\color{Dandelion}}{\hspace*{.5ex}\relax}

\newenvironment{yellow}{\relax\color{yellow}}{\hspace*{.5ex}\relax}

\newcommand{\beb}{\begin{blue}}
\newcommand{\eb}{\end{blue}}

\newcommand{\bey}{\begin{yellow}}
\newcommand{\ey}{\end{yellow}}

\newcommand{\ber}{\begin{red}}
\newcommand{\er}{\end{red}}

\newcommand{\berE}[1]{\begin{red}{}\marginnote{\fbox{\scshape\lowercase{E}}}%
#1}  

\renewcommand{\le}{\leqslant}
\renewcommand{\ge}{\geqslant}

\theoremstyle{plain}
\newtheorem{lemma}{Lemma}[section]
\newtheorem{prop}[lemma]{Proposition}
\newtheorem{theorem}[lemma]{Theorem}

\newcommand{\Prop}{\begin{prop}}
\newcommand{\enprop}{\end{prop}}
\newcommand{\Lemma}{\begin{lemma}}
\newcommand{\enlemma}{\end{lemma}}
\newcommand{\Th}{\begin{theorem}}
\newcommand{\enth}{\end{theorem}}
\newtheorem{corollary}[lemma]{Corollary}
\newcommand{\Cor}{\begin{corollary}}
\newcommand{\encor}{\end{corollary}}
\newtheorem{definition}[lemma]{Definition}
\newtheorem*{conjecture}{Conjecture}
\newcommand{\Def}{\begin{definition}}
\newcommand{\edf}{\end{definition}}
\newtheorem{sublemma}[lemma]{Sublemma}
\newcommand{\Sublemma}{\begin{sublemma}}
\newcommand{\ensub}{\end{sublemma}}

\theoremstyle{definition}
\newtheorem{remark}[lemma]{Remark}
\newtheorem{example}[lemma]{Example}
\newtheorem{Convention}[lemma]{Convention}
\newcommand{\Conv}{\begin{Convention}}
\newcommand{\enconv}{\end{Convention}}
\newcommand{\Conj}{\begin{conjecture}}
\newcommand{\enconj}{\end{conjecture}}
\newcommand{\Rem}{\begin{remark}}
\newcommand{\enrem}{\end{remark}}
\newcommand{\C}{{\mathbb C}}
\newcommand{\Q}{\mathbb {Q}}

\newcommand{\Z}{{\mathbb Z}}

\newcommand{\D}{\mathscr{D}}

\newcommand{\q}{{\mathfrak{q}}}

\newcommand{\one}{{\bf{1}}}

\newcommand{\seteq}{\mathbin{:=}}

\newcommand{\sh}{\operatorname{sh}}
\newcommand{\hd}{{\operatorname{hd}}}

\newcommand{\g}{{\mathfrak{g}}}

\newcommand{\M}{{\mathscr M}}

\newcommand{\hs}{\hspace*}
\newcommand{\ms}{\mspace}

\newcommand{\on}{\operatorname}

\newcommand{\bni}{\be[label=\rm(\roman*)]}
\newcommand{\bnum}{\bni}
\newcommand{\bna}{\be[label=\rm(\alph*)]}

\newcommand{\ba}{\begin{array}}
\newcommand{\ea}{\end{array}}

\newcommand{\eqsub}{\begin{subequations}\begin{eqnarray}}
\newcommand{\eneqsub}{\end{eqnarray}\end{subequations}}

\newcommand{\ol}{\overline}

\newcommand{\ko}{{{\mathbf{k}}}}

\nc{\la}{\lambda}
\nc{\lam}{\lambda}
\nc{\U}[1][\g]{U_q(#1)}
\nc{\te}{\tilde{e}}
\nc{\tem}{\tilde{e}^{\mathrm{max}}}
\nc{\tei}{\tilde{e}_i}
\nc{\tf}{\tilde{f}}
\nc{\tfm}{\tilde{f}^{\mathrm{max}}}
\nc{\tfi}{\tilde{f}_i}
\nc{\tU}{\widetilde U_q(\g)}
\nc{\tE}{\widetilde{E}}
\nc{\tF}{\widetilde{F}}
\nc{\tK}{\widetilde{K}}
\nc{\tEs}{\widetilde{E}^*}
\nc{\tFs}{\widetilde{F}^*}

\nc{\ttE}{\widetilde{\mathcal{E}}}
\nc{\ttF}{\widetilde{\mathcal{F}}}
\nc{\ttEs}{\ttE^*}
\nc{\ttFs}{\ttF^*}
\nc{\tfs}{\tf^*}
\nc{\tfss}[1]{\tf^{* \hskip 0.05em #1}}
\nc{\tes}{\te^*}
\nc{\tesm}{\tilde{e}^{* \hskip 0.05em \mathrm{max}}}
\nc{\tfsm}{\tilde{f}^{* \hskip 0.05em \mathrm{max}}}

\nc{\tk}{\tilde{k}}
\nc{\tkone}{\tk_{\ol{1}}}
\nc{\teone}{\tilde{e}_{\ol{1}}}
\nc{\tfone}{\tilde{f}_{\ol{1}}}

\nc{\teibar}{\tilde{e}_{\ol{i}}} \nc{\tfibar}{\tilde{f}_{\ol{i}}}
\nc{\tki}{{\tk}_{\ol {i}}}

\nc{\BZ}{{\mathbb{Z}}}
\nc{\al}{\alpha}
\nc{\qs}{{q}}
\nc{\lan}{\langle}
\nc{\ran}{\rangle}
\nc{\re}{{\mathrm{re}}}
\nc{\wt}{\operatorname{wt}}
\nc{\hwt}{\widehat{\wt}}
\nc{\Ht}{\mathrm{ht}}
\nc{\hHt}{\widehat{\Ht}}
\nc{\ch}{\operatorname{ch}}
\nc{\Um}[1][\g]{U^-_q(#1)}
\nc{\Ue}{U^+_q(\g)}
\nc{\ep}{\varepsilon}
\nc{\hep}{\widehat{\ep}}
\nc{\vphi}{\varphi}
\nc{\sphi}{\varphi^*}
\nc{\eps}{\ep^*}
\nc{\heps}{\hep^{ \hskip 0.2em  *}}

\nc{\nn}{\nonumber}
\def\max{{\mathop{\mathrm{max}}}}
\nc{\vp}{\varpi}
\nc{\cls}{{\operatorname{cl}}}
\nc{\Wt}{{\operatorname{Wt}}}
\nc{\Us}{U'_q(\g)}
\nc{\La}{\Lambda}
\nc{\tLa}{\widetilde\Lambda}
\nc{\ro}{{\rm(}}
\nc{\rf}{{\rm)}}
\nc{\norm}{{\mathrm{norm}}}
\nc{\qbox}{\quad\mbox}
\nc{\braid}{{\mathfrak{B}}}
\nc{\Ad}{\operatorname{Ad}}
\nc{\Aut}{\operatorname{Aut}}
\nc{\dt}[1]{\tilde{\tilde #1}}
\nc{\Sn}{S^{{\mathrm{norm}}}}
\nc{\aff}{{\mathrm{aff}}}
\nc{\rk}{{\mathrm{rk}}}
\nc{\tP}{\widetilde{P}}
\nc{\tW}{\widetilde{W}}
\nc{\Dyn}{\mathrm{Dyn}}
\nc{\tD}{\widetilde{\Delta}}
\nc{\height}[1]{{\operatorname{ht}}(#1)}
\nc{\bl}{\bigl(}
\nc{\br}{\bigr)}
\nc{\Hecke}{\mathrm{H}}
\nc{\HA}{\Hecke^{\mathrm{A}}}
\nc{\HB}{\Hecke^{\mathrm{B}}}
\newcommand{\scbul}{{\,\raise1pt\hbox{$\scriptscriptstyle\bullet$}\,}}
\nc{\vac}{{\phi}}
\nc{\Bt}{\B_\theta(\g)}
\nc{\be}{\begin{enumerate}}
\nc{\ee}{\end{enumerate}}
\nc{\low}{{\mathrm{low}}}
\nc{\upper}{{\mathrm{up}}}
\nc{\Zodd}{\Z_{\mathrm{odd}}}
\nc{\Ft}[1][n]{\mathbb{P}\mathrm{ol}_{#1}}
\nc{\Ftf}[1][n]{\widetilde{\mathbb{P}\mathrm{ol}}_{#1}}
\nc{\KA}{\on{K}^{\mathrm{A}}}
\nc{\KB}{\on{K}^{\mathrm{B}}}
\nc{\Res}{\on{Res}}
\nc{\Fc}[1][{n,m}]{\mathbf{F}_{#1}}
\nc{\tphi}{\tilde{\varphi}}
\nc{\CO}{\mathscr{O}}
\nc{\inte}{\mathrm{int}}
\nc{\Oint}{\mathcal{O}^{\ge0}_{\inte}}
\nc{\vs}{\vspace*}
\nc{\tLt}{\widetilde{L}}
\nc{\tL}{\widetilde{\Lambda}}
\nc{\tu}{\tilde{u}}
\nc{\noi}{\noindent}
\nc{\heigh}{\mathfrak{t}}
\nc{\lowest}{\mathfrak{l}}
\nc{\rootl}{\mathsf{Q}}
\nc{\rlQ}{\rootl}
\nc{\cl}{\colon}

\nc{\uqpg}{U'_q(\mathfrak g)}
\nc{\uq}{\uqpg}
\nc{\Oh}{\widehat{\mathcal{O}}}
\nc{\pn}{p_{\mathfrak{n}}}

\nc{\KLR}{KLR algebra}
\nc{\KLRs}{KLR algebras}
\nc{\cor}{\mathbf{k}}
\nc{\cora}{{\cor(A)}}
\nc{\haut}{\mathrm{ht}}
\nc{\tens}{\mathop\otimes}
\nc{\gmod}{\mbox{-$\mathrm{gmod}$}}
\nc{\gMod}{\mbox{-$\mathrm{gMod}$}}
\nc{\proj}{\mbox{-$\mathrm{proj}$}}
\nc{\gproj}{\mbox{-$\mathrm{gproj}$}}
\nc{\smod}{\mbox{-$\mathrm{mod}$}}
\nc{\Mod}{\mbox{-$\mathrm{Mod}$}}
\nc{\h}{\mathfrak h}
\nc{\Rnorm}{R^{\mathrm{norm}}}
\nc{\Runiv}{R^{\mathrm{univ}}}
\nc{\Rren}{R^{\mathrm{ren}}}

\nc{\Vhat}{\widehat{V}}
\nc{\F}{\mathcal{F}}

\def\T{{\mathcal T}}

\nc{\fd}[1][A]{\on{\mathrm{flat.dim}_{#1}}}
\nc{\bP}{{\mathbb{P}}}
\nc{\bPh}{\widehat{\mathbb{P}}}
\nc{\bK}[1][{n}]{\widehat{\mathbb{K}}_{#1}}
\nc{\bV}[1][{n}]{\widehat{V}^{\otimes{#1}}}
\nc{\bVK}[1][{n}]{\widehat{V}^{\otimes{#1}}_{\widehat{\mathbb{K}}}}
\nc{\hV}{\widehat{V}}
\nc{\opp}{\mathrm{opp}}
\nc{\col}{\colon}
\nc{\oep}{\epsilon}
\nc{\qtext}{\quad\text}
\nc{\qtextq}[1]{\quad\text{#1}\quad}
\nc{\longtwoheadrightarrow}[1][]{\xymatrix{\ar@{->>}[r]^-{{#1}}&}}
\nc{\epiTo}[1][]{\longtwoheadrightarrow[{#1}]}
\nc{\epito}{\twoheadrightarrow}
\nc{\monoTo}[1][]{\xymatrix{\ar@{>->}[r]^-{{#1}}&}}
\nc{\sym}{\mathfrak{S}}
\nc{\inp}[1]{{({#1})_{\mathrm{n}}}}
\nc{\rtl}{\rootl}
\nc{\wtd}{\widetilde}
\nc{\etens}{\boxtimes}
\nc{\ds}[1]{\mathrm{d}(#1)}
\nc{\rmat}[1]{{\mathbf{r}}_%
{\mspace{-2mu}\raisebox{-.6ex}{${\scriptstyle{#1}}$}}}
\nc{\rmats}[1]{{\mathbf{r}}_%
{\mspace{-2mu}\raisebox{-.6ex}{${\scriptscriptstyle{#1}}$}}}
\nc{\shc}{\mathcal{C}}
\nc{\shs}{\mathcal{S}}
\nc{\Fct}{{\on{Fct}}}
\nc{\tC}{\widetilde{\shc}}
\nc{\Zp}{\Z_{\ge0}}
\nc{\tPhi}{\widetilde{\Phi}}
\nc{\tT}{{\widetilde{\T}}}
\nc{\Ob}{\on{Ob}}
\nc{\bwr}{\mbox{\large$\wr$}}
\nc{\Img}{\on{Im}}
\nc{\Ab}{\mathcal{A}^{\mathrm{big}}}
\nc{\Sb}{\mathcal{S}^{\mathrm{big}}}
\nc{\As}{\mathcal{A}}
\nc{\Ss}{\mathcal{S}}
\nc{\ntens}{\widetilde{\otimes}}
\nc{\hR}{\widehat{R}}
\nc{\nconv}{\mathop{\mbox{\large $\odot$}}}
\nc{\snconv}{\mbox{\scriptsize$\odot$}}
\nc{\ts}{\tilde{s}}
\nc{\sho}{\mathcal{O}}
\nc{\bc}{\begin{cases}}
\nc{\ec}{\end{cases}}
\nc{\slnh}{{\widehat{\mathfrak{sl}}_N}}
\nc{\UA}{U_q'(\slnh)}
\nc{\KR}{R_K}
\nc{\cQ}{\mathcal{Q}}
\nc{\Irr}{\mathcal{I}rr}
\nc{\tQ}{\widetilde{\cQ}}
\nc{\bs}{\mathbf{s}}
\nc{\bL}{\mathbb{L}}
\nc{\tg}{\tilde{g}}

\nc{\conv}{\mathbin{\mbox{\large $\circ$}}}
\nc{\shconv}{\mathbin{\large\diamond}}
\nc{\sconv}{\mathbin{\large\Delta}}
\nc{\stens}{\mathbin{\large\Delta}}
\nc{\hconv}{\mathbin{\nabla}}
\nc{\htens}{\mathbin{\nabla}}

\nc{\Rm}{R^{\mathrm{ren}}}

\nc{\bQ}{\ol{Q}}

\nc{\de}{\on{\textfrak{d}\ms{1mu}}}

\nc{\xmono}{\ar@{>->}}
\nc{\xepi}{\ar@{->>}}
\nc{\db}[1]{\raisebox{-.5ex}[2ex][1.8ex]{$#1$}}
\nc{\wb}[1]{\mbox{$\rule[-1.1ex]{0ex}{2ex}#1$}}
\nc{\univ}{\mathrm{univ}}
\nc{\rM}{{}^*\mspace{-2mu}M}
\nc{\lM}{M^*}
\nc{\uqm}{\uq\smod}
\nc{\tR}{\widetilde{R}_{\gamma,\beta}}
\nc{\tx}{\tilde{x}}
\nc{\bi}{\mathbf{i}}
\nc{\ttau}{\widetilde{\tau}}

\nc{\tEnd}{\on{\widetilde{E}nd}}
\nc{\tHom}{\on{\widetilde{H}om}}

\nc{\K}{{J}}
\nc{\Kex}{{\K}_{\mathrm{ex}}}
\nc{\Kfr}{{\K}_{\mathrm{f\mspace{.01mu}r}}}
\nc{\coro}{\cor}
\nc{\tB}{\widetilde{B}}
\nc{\seed}{\mathscr{S}}

\nc{\up}{\mathrm{up}}
\nc{\bfa}{\mathbf{a}}
\nc{\bfb}{\mathbf{b}}
\nc{\bfc}{\mathbf{c}}
\nc{\bfe}{\mathbf{e}}

\newlength{\mylength}
\nc{\ov}[1]{\overline{#1}}
\nc{\Wlmj}[3]{\W_{#2,#3}^{(#1)}}
\nc{\Mkl}[2]{\M_\ttww(#1,#2)}
\nc{\mqs}{(-q^2)}
\nc{\Cquiver}{\upsigma}

\nc{\mut}[1]{{\mu}_{\mspace{-2mu}\raisebox{-.5ex}{${\scriptstyle{#1}}$}}}

\nc{\Kt}{\mathcal K_t}
\nc{\KT}{\mathbb{K}_t}
\nc{\yim}{y_{i,m}}
\nc{\yjm}{y_{j,m}}
\nc{\yjp}{y_{j,p}}
\nc{\yimp}{y_{i,m+1}}
\nc{\yjmp}{y_{j,m+1}}

\nc{\Refl}{\mathscr{S}}
\nc{\Reflinv}{{\Refl}^{-1}}

\nc{\catC}{\mathscr C}
\nc{\catA}{\mathcal A}
\nc{\shift}{\mathsf{sh}}
\nc{\rE}{ \mathsf{E} }
\nc{\rW}{ \mathcal{W} }
\nc{\rES}{ \mathcal{E} }

\nc{\brd}{\sigma} 
\nc{\into}{\xymatrix@C=3ex{{}\ar@{^{(}->}[r]&{}}}
\nc{\dual}{\D\ms{1.5mu}}
\nc{\cdual}{\mathsf{D}}
\nc{\cat}[1][{\g}]{\catC_{#1}^0}

\nc{\catCO}{{\catC_\g^0}}
\nc{\catCP}{{\catC_\g^{+}}}
\nc{\catCOD}{{\catC_\g^{0, \ddD}}}
\nc{\catCQ}{{\catC_{\qQ}}}
\nc{\catCQd}{{\catC_{\widetilde{\qQ}}}}
\nc{\catCD}{{\catC_{\ddD}}}
\nc{\catCDK}{{\catC_{\ddD, \iK}}}
\nc{\Li}{{\La^\infty}}
\nc{\sig}{{\sigma(\g)}}
\nc{\sigZ}{{\sigma_0(\g)}}
\nc{\sigQ}{{\sigma_\qQ(\g)}}
\nc{\sigQd}{{\sigma_{\widetilde{\qQ}}(\g)}}
\nc{\phiQd}{\phi_{\widetilde{\qQ}}}
\nc{\sigD}{{\sigma_\ddD(\g)}}

\nc{\ZZ}{{\mathbf{Z}}}
\nc{\sP}{{\mathsf{P}}}
\nc{\sV}{{\mathsf{V}}}
\nc{\rxw}{{\underline{w}}}
\nc{\boten}[1]{\overrightarrow{\bigotimes_{#1}}}
\nc{\cmA}{{\mathsf{A}}}
\nc{\cmC}{{\mathsf{C}}}
\nc{\ddD}{{\mathcal{D}}}
\nc{\ddDQ}{{\ddD_Q}}
\nc{\ddDQd}{{\ddD_{\widetilde{Q}}}}
\nc{\qQ}{{\mathcal{Q}}}
\nc{\gf}{{\g_{\mathrm{fin}}}}
\nc{\Df}{{\Delta_{\mathrm{fin}}}}
\nc{\If}{{I_{\mathrm{fin}}}}
\nc{\cmAf}{{\cmA_{\mathrm{fin}}}}
\nc{\weyl}{{\mathsf{W}}}
\nc{\weylf}{{\mathsf{W}_{\mathrm{fin}}}}
\nc{\sg}{{\mathfrak{S}}}
\nc{\weylA}{{\mathsf{W}_\cmA}}
\nc{\weylC}{{\mathsf{W}_\cmC}}
\nc{\Deg}{\mathrm{Deg}}
\nc{\Di}{\Deg^\infty}
\nc{\KRc}{{K_{q=1}(R_\cmC\gmod)}}
\nc{\prD}{{\Delta^+}}
\nc{\prDf}{{\Delta^+_{\mathrm{fin}}}}
\nc{\nrD}{{\Delta^-}}
\nc{\prDA}{{\Delta^+_\cmA}}
\nc{\prDC}{{\Delta^+_\cmC}}
\nc{\nrDC}{{\Delta^-_\cmC}}
\nc{\n}{{\mathfrak{n}}}
\nc{\Rt}{\mathsf{L}} 
\nc{\Cp}{\mathsf{V}} 
\nc{\cuspS}{{\mathsf{S}}}
\nc{\st}[1]{\{{#1}\}}
\nc{\bst}[1]{\bigl\{{#1}\bigr\}}
\nc{\WS}{Quantum affine Schur-Weyl duality\xspace}
\nc{\CWS}{Quantum affine Weyl-Schur duality}
\nc{\zz}{{{\mathbf{z}}}}
\nc{\wlP}{\mathsf{P}}
\nc{\wl}{\wlP}
\nc{\clp}{{\mathrm{cl}}}
\nc{\wlPc}{{\wlP_\clp}}
\nc{\awlP}{\widehat{\mathsf{P}}}
\nc{\dM}{\mathsf{M}}
\nc{\dC}{\mathsf{C}}
\nc{\cC}{\mathcal{C}}
\nc{\lR}{\widetilde{{R}}}
\nc{\zero}{\mathrm{zero}}
\nc{\PD}{principal }

\nc{\prtl}[1][J]{\rootl_{#1}^+}
\nc{\hL}{\widehat{\Rt}}
\nc{\hF}{\widehat{\F}}
\nc{\Proof}{\begin{proof}}
\nc{\QED}{\end{proof}}
\nc{\e}{\mathrm{e}}

\nc{\Aff}{\mathrm{Aff}}
\nc{\rT}{\mathcal{T}}		
\nc{\rr}{rationally renormalizable\xspace}
\nc{\RA}{{R_\cmA}}		
\nc{\RC}{{R_\cmC}}		
\nc{\proolim}[1][]{\mathop{``{\varprojlim}{\mbox{''}}}\limits_{#1}}
\nc{\qtq}[1][\text{and}]{\quad\text{#1}\quad}

\newcommand{\iB}{{{_i}B(\infty)}} 			 
\newcommand{\Bi}{{B_i (\infty)}} 			 
\newcommand{\cT}{\mathrm{S}} 		
\newcommand{\cTs}{\mathrm{S}^{\ms{1mu}*}}
\newcommand{\tcT}{\widetilde{\mathrm{S}}} 		 

\nc{\corh}{\widehat{\cor}}
\nc{\ang}[1]{\langle{#1}\rangle}
\nc{\rc}{renormalizing coefficient\xspace}
\nc{\cz}{{\cor[z^{\pm1}]}}
\nc{\tp}{\ms{1.5mu}{\widetilde{p}}\ms{2mu}}
\nc{\G}{\mathcal{G}}
\nc{\cc}{\mathfrak{c}}
\nc{\rsP}{{\Phi_\g}}
\nc{\rsX}{{X_\g}}
\nc{\rs}{ \mathsf{s} }
\nc{\Dynkin}{\mathrm{D}}
\nc{\Dat}{\sigma}
\nc{\hf}{\xi}

\nc{\cB}{\widehat{B}}
\nc{\iK}{\mathsf{K}}
\nc{\cBg}[1][\g]{\widehat{B}_{#1}(\infty)}
\nc{\cBsg}[1][\g]{\widehat{B}_{#1}(\infty)^*}
\nc{\cBgk}[1][\g]{\widehat{B}_{#1}^{\iK}(\infty)}
\nc{\cb}{\mathbf{b}}
\nc{\cI}{ \widehat{I} }
\nc{\cIf}{{\widehat{I}_{\mathrm{fin}}}}
\nc{\cIz}{{\widehat{I}_{0}}}
\nc{\cJ}{ \widehat{J} }
\nc{\sB}{\mathcal{B}}
\nc{\sBk}{\sB_{\iK}}
\nc{\sBdk}{\sB_{\ddD, \iK}}
\nc{\sBD}{\sB_{\ddD}}
\nc{\sBQ}{\sB_{\qQ}}
\nc{\sBkg}{\sBk(\g)}
\nc{\cs}{\star}

\nc{\cd}{\mathrm{D}}
\nc{\cm}{\mathbf{m}}
\nc{\MS}{\mathsf{MS}}
\nc{\hMS}{\widehat{\mathsf{MS}}}
\nc{\rS}{\mathbf{S}}
\nc{\rSs}{\rS^*}
\nc{\brS}{\overline{\rS}}

\nc{\crI}{ \mathsf{I}}
\nc{\crhI}{\mathscr{S} }
\nc{\crD}{\mathsf{D}}
\nc{\crc}{\mathsf{c}}
\nc{\crB}{\mathscr{P}_n}
\nc{\cru}{\mathsf{u}}
\nc{\crsh}{\mathsf{sh}}
\nc{\bfs}{\mathbf{s}}
\nc{\crBB}[1]{\mathscr{P}_#1}

\nc{\gc}{{\g_{\cmC}}}
\nc{\cL}{\mathcal{L}}
\nc{\cLD}{{\cL_\ddD}}
\nc{\ccL}{\mathscr{L}}
\nc{\ccLD}{{\ccL_\ddD}}
\nc{\clen}{\mathsf{len}}

\nc{\hv}{\mathsf{1}}
\nc{\qt}[1]{\quad\text{#1}}
\nc{\snoi}{\smallskip\noindent}
\nc{\mnoi}{\medskip\noindent}

\nc{\ul}[1]{\underline{#1}}
\nc{\dul}[1]{\underline{\underline{#1}}}
\nc{\tul}[1]{\underline{\underline{\underline{#1}}}}
\nc{\qh}{\qedhere}
\nc{\ca}{\mathsf{v}}
\nc{\sck}[1][k]{\ms{8mu}{\raisebox{-1.3ex}{$\scriptstyle{#1}$}\hs{-1.4ex}{\succ}}\ms{4mu}}
\nc{\scke}[1][k]{\ms{8mu}{\raisebox{-1.3ex}{$\scriptstyle{#1}$}\hs{-1.4ex}%
{\succcurlyeq}}\ms{4mu}}

\nc{\edot}{\emptyset}
\nc{\Nf}{N_{\gf}}
\nc{\fin}{\mathrm{fin}}
\nc{\trg}{\scalebox{.7}{$\triangle$}}
\nc{\ake}[1][1ex]{\rule[-#1]{0ex}{1ex}}
\nc{\akew}[1][1ex]{\rule[-1ex]{#1}{0ex}}
\nc{\akeu}[1][1ex]{\rule[#1]{0ex}{1ex}}

\nc{\bg}{\mathscr{B}}
\newcommand{\B}{B(\infty)}
\nc{\ipi}{{{_i}\pi}}
\nc{\pii}{{\pi_i}}
\nc{\Ld}{\mathcal{P}}
\nc{\Dc}{\Upsilon}
\nc{\DcL}{\mathcal{L}}
\nc{\bR}{\mathsf{R}}
\nc{\bRs}{\bR^*}
\nc{\Y}{\mathrm{Y}}
\nc{\CC}{\mathbb{C}}
\nc{\bS}{\mathsf{S}}
\nc{\x}{\mathbf{x}}

\nc{\vP}{\mathrm{P}}
\nc{\vD}{\Delta}
\nc{\bfi}{\mathbf{i}}
\nc{\bfj}{\mathbf{j}}
\nc{\bfv}{\mathbf{v}}
\nc{\bfm}{\mathbf{m}}
\nc{\iGr}{\mathbb{G}}
\nc{\biGr}{ \overline\iGr}
\nc{\gGr}{\widetilde{\mathrm{Gr}}^\circ }
\nc{\tGr}{\widetilde{\mathrm{Gr}}}
\nc{\tdGr}{\widetilde{\mathrm{Gr}}^{ \diamondsuit}}
\nc{\toGr}{\widetilde{\mathrm{Gr}}^\circ}
\nc{\bT}{\mathsf{T}}
\nc{\qT}{\mathcal{Y}_t}
\nc{\qAg}{\widehat{\mathcal{A}}_\qq}
\nc{\sfb}{\mathsf{b}}
\nc{\cvC}[1]{\mathcal{C}(#1)}
\nc{\cvF}[1]{\mathcal{F}(#1)}

\nc{\tsig}{\tilde{\sigma}}
\nc{\eF}{\equiv_{\mathcal F}}
\nc{\ie}{\mathsf{e}}
\nc{\qq}{\mathtt{q}}
\nc{\qK}{\mathbb{K}}
\nc{\pk}{\mathsf{p}}
\nc{\bC}{\mathsf{C}}
\nc{\bbA}{\mathbb{A}}
\newcommand{\bnom}[1]{\begin{bmatrix}#1\end{bmatrix}}
\nc{\vPab}[1]{\vP_{#1}}
\nc{\SSYT}{{\rm SSYT}}
\nc{\bbS}{\overline{\mathrm{S}}}
\nc{\gB}{\mathscr{B}}
\nc{\gt}{\mathsf{t}}
\nc{\KG}{\Omega}

\newcommand{\oprod}{\prod\limits^{{\scriptstyle\xrightarrow{\akew[1ex]}}}}
\newcommand{\rprod}{\prod^{\xleftarrow{}}}

\title[Grassmannians and extended crystals of type $A$]
{Braid group actions on grassmannians and extended crystals of type $A$}

\author[J.-R. Li]{Jian-Rong Li}
\thanks{The research of J.-R. Li was supported by the Austrian Science Fund (FWF): P-34602, Grant DOI: 10.55776/P34602, and PAT 9039323, Grant-DOI 10.55776/PAT9039323.}
\address[J.-R. Li]{Faculty of Mathematics, University of Vienna, Oskar-Morgenstern-Platz 1, 1090 Vienna, Austria}
\email[J.-R. Li]{lijr07@gmail.com}

\author[E. Park]{Euiyong Park}
\thanks{The research of E.\ Park was supported by the  National Research Foundation of Korea(NRF) grant funded by the Korea government (MSIT) (RS-2023-00273425 and NRF-2020R1A5A1016126).}
\address[E. Park]{Department of Mathematics, University of Seoul, Seoul 02504, Korea}
\email[E. Park]{epark@uos.ac.kr}

\date{October 12, 2024}

\begin{document}

\begin{abstract}
Let $\sigma_i$ be the braid actions on infinite Grassmannian cluster algebras induced from Fraser's braid group actions. Let $\bT_i$ be the braid group actions on (quantum) Grothendieck rings of Hernandez-Leclerc category $\catCO$ of affine type $A_n^{(1)}$, and $\bR_i$
the braid group actions on the corresponding extended crystals.
In the paper,  we prove that the actions $\sigma_i$  coincide with the braid group actions $\bT_i$ and $\bR_i$. 
\end{abstract}

\maketitle

\setcounter{tocdepth}{1}
\tableofcontents

\section{Introduction}

Let $\catC_\g$ be the category of finite-dimensional integrable modules over a quantum affine algebra $U_q'(\g)$, where $q$ is an indeterminate. Hernandez and Leclerc introduced a distinguished subcategory $\catCO$ of $\catC_\g$ determined by certain fundamental modules determined by a $Q$-datum $\mathcal{Q}$ (see \cite{HL10, HL15} and see also \cite{FHOO22}). A cluster-algebraic approach to the Grothendieck ring $K(\catCO)$ of the category $ \catCO$ was introduced by Hernandez and Leclerc, which led to beautiful results on the structure of $K(\catCO)$ (see \cite{FHOO23, HL10, HL15, KKOP20, KKOP24B} and references therein).

The \emph{quantum Grothendieck ring} $K_t(\catCO)$ is a $t$-deformation of the Grothendieck ring $K(\catCO)$ which is realized as a subalgebra inside the \emph{quantum torus} via \emph{$(q, t)$-characters} of modules in $\catCO$ (\cite{Her04, Nak04, VV03}). A ring presentation of $K_t(\catCO)$ was introduced and studied in \cite{HL15}. This formal algebra defined by the ring presentation is denoted by $\qAg$, which is called the \emph{bosonic extension} (see Section \ref{Sec: bosonic ext}).  
Here we set $\qq := t^{-1}$ for consistency with the previous works \cite{KKOP21, KKOP24B, OP24} of the second named author. 
There exist braid group actions $T_i$ on $\qAg$ (\cite{KKOP21, JLO23B, KKOP24B}), which take a crucial role in \emph{PBW theory} for $\qAg$ introduced in \cite{OP24, KKOP24B}. It is also proved in \cite{KKOP24B} that $T_i$ can be understood as lifts of the braid group actions $R_i$ on the \emph{extended crystals} $\widehat{B}(\infty) $ (see \cite{KP22, P23} and see also Section \ref{Sec: extended crystal}). Let $\sB(\g)$ be the set of the isomorphism classes of simple modules in $\catCO$. It is proved in \cite{KP22} that $\sB(\g)$ has the \emph{categorical crystal structure}, which is isomorphic to the extended crystal $\cB(\infty)$.

Given a $Q$-datum $\qQ$, there exists the corresponding isomorphism  
\begin{align*}
\Phi_\qQ: \qAg \buildrel \sim \over \longrightarrow \qK, 
\end{align*}
where $\qK $ is the scalar extension $\Q(t^{1/2}) \tens_{\Z[t^{\pm 1}]}  K_t(\catCO)$ of the quantum Grothendieck ring $ K_t(\catCO)$. 
Via the isomorphism $\Phi_\qQ$, one can obtain the braid group actions $\bT_i^\qQ$ on $\qK$ induced from $T_i$. The braid group actions $\bT_i^\qQ$  preserve the $(q,t)$-characters of simple modules in $\qK$ in terms of \emph{global basis} of $\qAg$ (\cite{KKOP24A,KKOP24C}). This means that $\bT_i^\qQ$ can be viewed as  permutations of the set of simple modules in $\catCO$.
Thus $\bT_i^\qQ$ send prime (resp.\ real) elements to prime (resp.\ real) elements. 
It turns out that $\bT_i^\qQ$ coincide with the braid group actions $\bR_i^\qQ$ on the categorical crystal $\sB(\g)$, which arise from $R_i$ on $\cB(\infty)$ via the isomorphism $\phi_\qQ:\cB(\infty) \buildrel \sim \over \longrightarrow \sB(\g)$ determined by $\qQ$. 

In affine type $A_n^{(1)}$, the Grothendieck ring $K(\catCO)$ has an interesting connection to the \emph{Grassmannian cluster algebras} (\cite{Sco06}). 
The ring $ \C \otimes_\Z K(\catC_\ell)$ of the subcategory $\catC_\ell$ (\cite{HL10}) of $\catCO$ is isomorphic to the quotient of the cluster algebra $\C[\tGr(n+1, n+\ell+2)]$ by specializing certain frozen variables at $1$ (see \cite{HL10} and see also \cite{CDFL20}). We call this quotient ring the \emph{Grassmannian cluster algebra}. Thus, the Grassmannian cluster algebra is categorified by the subcategory $\catC_\ell$ in monoidal sense. 
Fraser \cite{Fra20} introduced braid group actions $\tsig_i$ on a certain localization $\C[\tGr^\circ (n+1, n+\ell+2)]$. These braid group actions give an important tool to study Grassmannian cluster algebras because braid group actions send cluster variables to cluster variables (up to multiplying certain frozen variables). 
One can consider the actions $\sigma_i$ induced from  $\tsig_i$ on \emph{infinite Grassmannian cluster algebras} as taking $\ell \rightarrow \infty$.
Since $\catCO$ can be viewed as a limit of $\catC_\ell$ as $\ell \rightarrow \infty$, the braid groups actions $\sigma_i$ also can be interpreted as actions on $K(\catCO)$. 

In the paper, we show that the braid group actions $\sigma_i$ on infinite Grassmannian cluster algebras coincide with the braid group actions $\bT_i^\qQ$ on (quantum) Grothendieck rings and $\bR_i^\qQ$ on extended crystals. 
Since $\bR_i^\qQ$ is described combinatorially in terms of crystals, this result leads us to a way to compute the braid group actions on (infinite) Grassmannian cluster algebras in the viewpoint of the crystal theory.

Let us explain the result in more details. Let $m \in \Z_{>1}$ and define
\begin{align*}
\cvC{m}&:= \{  (i_1,i_2, \ldots, i_m) \in \Z^{ m} \mid i_1 < i_2 < \ldots < i_m \}, \\
\cvF{m}&:= \{ (i, i+1 \ldots, i+m-1 ) \in \Z^{m} \mid i\in \Z \}. 
\end{align*}
For each $\bfi = (i_1, i_2, \ldots, i_m) \in \cvC{m}$, we set $ \vP_\bfi = \vP_{i_1, i_2, \ldots, i_m} $ to be an indeterminate. We set $\widetilde{A} := \C[\vP_\bfi, \ \vP_\bfj^{-1} \mid \bfi \in \cvC{m},\ \bfj \in \cvF{m}]$ and  define $\iGr_m$ to be the quotient algebra of $\widetilde{A}$ by the ideal generated by Pl\"{u}cker relations (\ref{Eq: Plucker}). Denote by $\biGr_m$ the quotient of $\iGr_m$ by the ideal generated by 
$ \vP_{\bfi} - 1 $ for all $\bfi \in \cvF{m}$. It is shown that the algebra $\biGr_m$ is the infinite Grassmannian cluster algebra by investigating the localization $\C[\tdGr(m, N)]$ (see Lemma \ref{Lem: phimN} and \eqref{Eq: diagram}). 
We study Fraser's braid group action $\tsig_i$ ($i\in [1,m-1]$) on the localization of $\C[ \toGr(m,N)]$ when $m$ divides $N$, and extend the action $\tsig_i$ to the infinite Grassmannian cluster algebra $\biGr_m$, which are denoted by $\sigma_i$ (see Proposition \ref{prop: sigma_i} and \eqref{Eq: sigmai}). Remark \ref{Rmk: inf Gr} explains the braid group actions $\sigma_i$  in terms of (infinite) configuration of vectors.

We choose and fix the $Q$-datum $\qQ$ of type $A_n$ defined in \eqref{Eq: An Q-datum}. We simply write $\bT_i := \bT_i^\qQ$ and $\bR_i := \bR_i^\qQ $ for any $i\in I:= [1,n]$. Since $\bT_i$ respect to the set of $(q,t)$-characters of simple modules, $\bT_i$ can be regarded as an automorphisms of the Grothendieck ring $K(\catCO)$ by specializing at $t=1$.  
We regard the categorical crystal $\sB(\g)$ as a subset of $K(\catCO)$. 
We then obtain the isomorphism 
\begin{align*} 
\KG_n : K(\catCO)  \buildrel \sim \over \longrightarrow \biGr_{n+1} 
\end{align*}
defined in \eqref{Eq: Phi} and both of $K(\catCO)$ and $\biGr_{n+1} $ have the braid group actions $\bT_i$ and $\sigma_i$ respectively.
By computing explicitly the image of certain Pl\"{u}cker variables under the action $\tsig_1$ (Lemma \ref{Lem: tsig}),  
we prove that $ \KG_n \circ \bT_i = \sigma_i \circ \KG_n $, i.e., the following diagram commutes:
$$
\xymatrix{
K(\catCO) \ar[d]_{\bT_i}  \ar[rr]^\sim_{\KG_n} && \biGr_{n+1} \ar[d]^{\sigma_i} \\
K(\catCO) \ar[rr]^\sim_{\KG_n} && \biGr_{n+1},
}
$$ 
(see Theorem \ref{Thm: T=sig}). Therefore the braid group actions $\sigma_i$ on $\biGr_{n+1}$ coincide with the braid group actions $\bT_i$ on $K(\catCO)$, and hence coincide with the braid group actions $\bR_i$ on $\mathcal{B}(\mathfrak{g})$ (see Corollary \ref{Cor: R=sig}).
In course of the proof, the \emph{multisegment realization} of $\cB(\infty)$ and the crystal description of $\bR_i$ are used crucially.
In rank $2$ cases, we give an explicit formula for braid group actions and compute an example in  viewpoint of primeness and reality (see Section \ref{sec:examples for rank 2 cases}). 

The paper is organised as follows. In Section \ref{sec:preliminaries}, we recall results of quantum affine algebras, Hernandez-Leclerc categories, Bosonic extensions, extended crystals, and categorical crystals. In Section \ref{Sec: ed An}, we give explicit descriptions for type $A_n$. In Section \ref{Sec: igca}, we define braid group actions on the infinite Grassmannian cluster algebra by extending the braid group actions introduced by Fraser. In Section \ref{sec:compare braid group actions}, we compare the braid group actions $\bT_i$, $\bR_i$ and $\sigma_i$. In Section \ref{sec:examples for rank 2 cases}, we illustrate the braid group actions by examples. 


\vskip 1em

\subsection*{Convention}
Throughout this paper, we use the following convention.
\bnum
\item For a statement $P$, we set $\delta(P)$ to be $1$ or $0$ depending on whether $ P$ is true or not. In particular, we set $\delta_{i,j}=\delta(i =j)$. 

\item For $a\in \Z \cup \{ -\infty \} $ and $b\in \Z \cup \{ \infty \} $ with $a\le b$, we set 
\begin{align*}
& [a,b] =\{  k \in \Z \ | \ a \le k \le b\}, &&  [a,b) =\{  k \in \Z \ | \ a \le k < b\}, \allowdisplaybreaks\\
& (a,b] =\{  k \in \Z \ | \ a < k \le b\}, &&  (a,b) =\{  k \in \Z \ | \ a < k < b\},
\end{align*}
and call them \emph{intervals}. 
When $a> b$, we understand them as empty sets. For simplicity, when $a=b$, we write $[a]$ for $[a,b]$.

\item For a totally ordered set $J = \{ \cdots < j_{-1} < j_0 < j_1 < j_2 < \cdots \}$, write
$$
\oprod_{j \in J} A_j \seteq \cdots A_{j_2}A_{j_1}A_{j_0}A_{j_{-1}}A_{j_{-2}} \cdots, \quad
\rprod_{j \in J} A_j \seteq \cdots A_{j_{-2}}A_{j_{-1}}A_{j_0}A_{j_{1}}A_{j_{2}} \cdots. 
$$

\ee

\vskip 2em

\section{Preliminaries} \label{sec:preliminaries}
In this section, we recall results of quantum affine algebras, Hernandez-Leclerc categories, Bosonic extensions, extended crystals, and categorical crystals. 
We refer the readers to references \cite{CP94, HL10, HL15, KP22, P23}.

\subsection{Quantum affine algebras of affine type $A_n^{(1)}$} 
We shall review briefly  the quantum affine algebra of type $A_n^{(1)}$ in this subsection. 

Let $I = \{0,1, \ldots, n \}$ and let $\cmA =(a_{i,j})_{i,j \in I}$ be a \emph{Cartan matrix} of affine type $A_n^{(1)}$.  
We shall use the standard convention for the Dynkin diagram of affine type $A_n^{(1)}$ in 
\cite{Kac}.
Let $\Pi=\{\al_i \}_{i \in I}$ be the set of simple roots,  $\Pi^\vee=\{ h_i \}_{i \in I}$ the set of simple coroots, and
$\wl$ the weight lattice of $\g$. 
We denote by $\prD$ the set of positive roots. The quintuple $(\cmA,\wl,\Pi,\wl^\vee,\Pi^\vee)$ is called the \emph{Cartan datum}. 
Let $\g$ be the \emph{affine Kac-Moody algebra} associated with $\cmA$. We set $I_0 := I \setminus \{0\}$, and denote by $\g_0$  the subalgebra of $\g$ generated by $f_i$, $e_i$ and $h_i$ for $i\in I_0$.

Let $\cor$ be the algebraic closure of $\Q(q)$ in $\cup_{t>0} \C((q^{1/t}))$, where $q$ is an indeterminate. We define 
$$
[n]_q= \dfrac{q^n-q^{-n}}{q-q^{-1}} \quad \text{ and }\quad  [n]_q! = \prod_{k=1}^n [k]_q.
$$

\begin{definition} \label{Def: GKM}
The {\em quantum affine algebra} $U_q(\g)$ associated with $(\cmA,\wl,\Pi,\wl^\vee,\Pi^\vee)$ is the associative algebra over $\ko$ with $1$ generated by $e_i,f_i$ $(i \in I)$ and
$q^{h}$ $(h \in  \wl^{\vee})$ satisfying following relations:
\bnum
\item  $q^0=1, q^{h} q^{h'}=q^{h+h'} $\quad for $ h,h' \in  \wl^{\vee},$
\item  $q^{h}e_i q^{-h}= q^{\langle h, \alpha_i \rangle} e_i$,
$q^{h}f_i q^{-h} = q^{-\langle h, \alpha_i \rangle }f_i$\quad for $h \in \wl^{\vee}, i \in I$,
\item  $e_if_j - f_je_i =  \delta_{ij} \dfrac{K_i -K^{-1}_i}{q_i- q^{-1}_i }, \ \ \text{ where } K_i=q^{ h_i},$
\item  $\displaystyle \sum^{1-a_{ij}}_{k=0}
(-1)^ke^{(1-a_{ij}-k)}_i e_j e^{(k)}_i =  \sum^{1-a_{ij}}_{k=0} (-1)^k
f^{(1-a_{ij}-k)}_i f_jf^{(k)}_i=0 \quad \text{ for }  i \ne j, $
\ee
where $e_i^{(k)}=e_i^k/[k]_q!$ and $f_i^{(k)}=f_i^k/[k]_q!$.
\end{definition}
The coproduct $\Delta$ of $U_q(\g)$ is defined by
\begin{align*} 
\Delta(q^h)=q^h \tens q^h, \ \ \Delta(e_i)=e_i \tens K_i^{-1}+1 \tens e_i, \ \Delta(f_i)=f_i \tens 1 +K_i \tens f_i.
\end{align*}

\begin{remark} \label{rmk: convention}
The coproduct $\Delta$ is the same as that of \cite{Kas02, KKOP20,KKOP23,KP22}. 
For the consistency with the previous work \cite{KP22} on categorical crystals by the second author, we basically follow the conventions and notations given in \cite{KKOP23, KP22}. There are some differences from \cite{HL15, FHOO22}.  

\end{remark}

We denote by  $U_q'(\g)$ the $\cor$-subalgebra of $U_q(\g)$ generated by $ e_i$, $f_i$ and $K_i^{\pm1}$  for $i\in I$.
Let $\catC_\g$ be the category of finite-dimensional \emph{integrable} $U_q'(\g)$-modules. The \emph{tensor product} $\tens$ gives a monoidal category structure on $\catC_\g$.

We denote by $\catC_\g$ the category of finite-dimensional integrable $U_q'(\g)$-modules.
Simple modules in $\catC_\g$ are parameterized by the set $(1+z\cor[z])^{I_0}$ of $I_0$-tuples of monic polynomials in an indeterminate $z$. These polynomials are called \emph{Drinfeld polynomials} (\cite{CP94}). 
Let $Y_{i,a}$ be indeterminates for $i\in I_0$ and $a \in \Z$.
We write a monomial $m $ as 
\begin{align*}
m = \prod_{(i,p) \in I_0 \times \Z  } Y_{i,p}^{u_{i,p}(m)}.
\end{align*}
A monomial $m$ is called \emph{dominant} if $u_{i,p}(m) \ge 0$ for all $(i,p)$.
For a dominant monomial $m $, there exists a simple module $V(m) \in \catC_\g$ associated with the Drinfeld polynomial $(\prod_p (1-q^pz)^{u_{i,p}(m)})_{i \in I_0}$.
For each $i\in I_0$, we call $V(Y_{i,a})$ the \emph{$i$-th fundamental module}. 
Note that the head of the ordered tensor product 
$$ 
V(Y_{i_1,a_1}) \otimes V(Y_{i_2,a_2}) \otimes \cdots  \otimes V(Y_{i_k,a_k}) \quad \text{($a_1 \ge a_2 \ge \cdots \ge a_k $)}
$$  
is isomorphic to the simple module $ V \left( \prod_{t=1}^k  Y_{i_t, a_t} \right)$ (see \cite{Kas02}). 

We write $\dual$ for the right dual functor in $\catC_\g$. Note that 
\begin{align} \label{Eq: dual V}
\dual (V(Y_{i,a}) ) \simeq V(Y_{n+1-i,a+n+1})
\end{align}
for any $i\in I_0$ (see \cite[Section 2]{KKOP20} for example).

\begin{remark} \label{Rmk: V(varpi)}
Let $V(\varpi_i)_x$ be the $i$-th fundamental module defined in \cite{Kas02}.
If we choose a function $o:I_0 \rightarrow \{1,-1\}$ such that $o(i) = -o(i-1)$ for $i\in [2,n]$, then 
$V(Y_{i,a})$ corresponds to the fundamental module $V(\varpi_i)_{o(i) (-1)^{a+1} (-q)^{a+n+1}}$ (see \cite[Remark 3.3]{Nak04A} and \cite[Remark 3.28]{HO19}).
Thus the fundamental modules $V(Y_{i,a})$ can be identified with $V(\varpi_i)_{(-q)^a}$ by shifting $(-q)^{n+1}$ with a suitable sign.
\end{remark}

\subsection{Hernandez-Leclerc categories} \label{Subsection: HL cat} 
We shall review briefly the Hernandez-Leclerc category $\catCO$ over the quantum affine algebra of type $A_n^{(1)}$.

Let $\Dynkin$ be the Dynkin diagram of type $A_n$, and let $\xi$ be a \emph{height function} of $\Dynkin$. 
The pair $\qQ = (\Dynkin, \xi)$ is called a \emph{$Q$-datum} of affine type $A_n^{(1)}$ (see \cite{FO21} for more details on $Q$-data). 
Let 
\begin{align*}
\crhI_{\Z} &:= \{ (i,a) \in I_0 \times \Z \mid i\in I_0, \  a \in 2\Z + \xi(i)  \}, \\
\crhI_{\Z_{\ge 0}} &:= \{ (i,a) \in I_0 \times \Z \mid i\in I_0, \  a \in 2\Z_{\ge0} + \xi(i)  \}.
\end{align*}
The \emph{Hernandez-Leclerc category} $\catCO$ (resp.\ $\catCP$) is defined to be the smallest full subcategory of $\catC_\g$ such that 
\bna
\item it  contains fundamental modules $V(Y_{i,a})$ for all $ (i,a) \in \crhI_\Z$ (resp.\ $ (i,a) \in 
 \crhI_{\Z_{\ge0}}$),
\item it is stable under taking subquotients, extensions and tensor products. 
\ee
To a $Q$-datum $\qQ$, one can associate a subset $\crhI_\qQ \subset \crhI_\Z$ such that there is a bijection 
$$ 
\theta_\qQ : \prD \buildrel \sim \over \longrightarrow  \crhI_\qQ
$$ satisfying certain conditions (see \cite{HL15} for details, and see also \cite{FO21} and \cite[Section 6.2]{KKOP23}).
The Hernandez-Leclerc category $\catCQ$ is defined to be the smallest full subcategory of $\catCO$ such that 
\bna
\item $\catCQ$  contains fundamental modules $V(Y_{i,a})$ for all $ (i,a) \in \crhI_\qQ$,
\item $\catCQ$ is stable under taking subquotients, extensions and tensor products. 
\ee
 Note that all simple modules of $\catCO$ (resp.\ $\catCP$, $\catCQ$) are parameterized by the set of dominant monomials in $Y_{i,a}$'s for $(i,a)\in \crhI_{\Z}$ (resp.\ $\crhI_{\Z_{\ge0}}$, $\crhI_{\qQ}$).
We denote by $K(\catCO)$ (resp.\ $K(\catCP)$, $K(\catCQ)$) the \emph{Grothendieck ring} of $ \catCO$ (resp.\ $ \catCP$, $ \catCQ$).

We define $\ddD_\qQ := \{ \Rt_i\}_{i\in I_0}$, where we set 
$$
\Rt_i := V(Y_{ \theta_\qQ(\al_i)} ) \qquad \text{ for $i\in I_0$.}
$$
It is shown in \cite[Proposition 6.5]{KKOP23} that $\ddD_\qQ$ becomes a \emph{complete duality datum} (see \cite[Definition 6.1]{KKOP23} for details).

\subsection{Bosonic extensions} \label{Sec: bosonic ext}

Let $\qT = \langle Y_{i,a} \mid (i,a) \in  \crhI_\Z \rangle$ be the \emph{quantum torus} over $\Z[t^{\pm 1/2}]$ associated with $\g_0$, and  for a simple module $L \in \catCO$, let $\chi_{q,t}(L) \in \qT$ be the \emph{$(q,t)$-character} of $L$, where $t$ is an indeterminate (see \cite{Her03, Her04, Nak04, VV03} and see also \cite[Section 5]{HL15} for details).

We define $K_t(\catCO)$ to be the $\Z[t^{\pm 1/2}]$-subalgebra of $\qT$ generated by $\chi_{q,t}(L)$ for all simple modules $L \in \catCO$. The algebra $K_t(\catCO)$ is called the \emph{quantum Grothendieck ring} of $\catCO$, which can be viewed as a $t$-deformation of the Grothendieck ring $K(\catCO)$. 
Define 
$$
\qK := \Q(t^{1/2}) \tens_{\Z[t^{\pm 1}]}  K_t(\catCO).
$$
It is shown in \cite{Nak04, VV03} that 
since $U_q'(\g)$ is of type $A_n^{(1)}$ 
when specializing at $t=1$, the $(q,t)$-characters correspond to the \emph{$q$-characters} introduced by Frenkel-Reshetikhin (\cite{FR99}).   
For the sake of simplicity, we denote by $[V]$ the $(q,t)$-character $\chi_{q,t}(V)$ of the simple module $V$. If no confusion arises,  we also use the same notation $[V]$ for the $q$-character of $V$ and regard it as an element of $K(\catCO)$.

We set 
$$ 
\qq := t^{-1}.
 $$
Let $\qAg$ be the $\Q(\qq^{1/2})$-algebra generated by the generators
$\{f_{i,p}\mid i\in I_0,p\in \Z\}$ with the defining relations:
\bna
\item 
$\sum_{k=0}^{1- \langle h_i,\al_j \rangle}(-1)^k\bnom{1-\langle{h_i,\al_j} \rangle\\k}_\qq
f_{i,p}^kf_{j,p}f_{i,p}^{1-\langle h_i,\al_j \rangle-k}=0$\quad for any $i,j\in I$ such that $i\not=j$ and $p\in \Z$,
\item 
$f_{i,m}f_{j,p}=\qq^{(-1)^{p-m+1}(\al_i,\al_j)}f_{j,p}f_{i,m}
+\delta_{(j,p),\,(i,m+1)}(1-\qq^2)$\quad if $m<p$.
\ee
The algebra $\qAg$ is called a \emph{Bosonic extension} of $U_\qq^-(\g_0)$. 
It is proved in \cite[Theorem 7.3]{HL15} that, for a given $Q$-datum $\qQ$, there is a $\Q(\qq^{1/2})$-algebra isomorphism  
\begin{align} \label{Eq: Phi_Q}
\Phi_\qQ: \qAg \buildrel \sim \over \longrightarrow \qK 
\end{align}
mapping $ f_{i,m} $ to $[ \dual^m (\Rt_i))]$ for any $i\in I_0$ and $m\in \Z$.
The algebra $\qAg$ has an action $T_i$ of the braid group $\bg_n$ associated with the Cartan matrix $\cmA$ (\cite[Theorem 2.3]{KKOP21} and see also \cite[Section 8.1]{JLO23B}). For any $i\in I_0$, the automorphism $T_i: \qAg \longrightarrow \qAg$ is defined by 
\begin{align} \label{Eq: T_i}
	T_i(f_{j,p}) \seteq
	\begin{cases}
		f_{j,p+\delta_{i,j}} & \text{ if } a_{i,j} \ge 0,\\
		\displaystyle \frac{   \sum_{r+s = 1}  (-1)^{ r} \qq^{ -1/2 + r} f_{i,p}^{(s)} f_{j,p} f_{i,p}^{(r)}   }{ \qq^{-1} - \qq} & \text{ if } a_{i,j} = -1.
	\end{cases}
\end{align}
Let $*: \qAg \buildrel \sim \over\longrightarrow \qAg$ be the $\Q(\qq^{1/2})$-algebra anti-automorphism defined by $f_{i,m} ^* = f_{i,-m}$ for any $i\in I_0$ and $m\in \Z$. Note that $T_i^* := * \circ T_i \circ * $ is the inverse of $T_i$ (\cite[Lemma 4.1]{OP24}).
Then, for each $i\in I_0$, the braid group actions $\bT_i^\qQ$ on $\qK$ is defined by the following commutative diagram:
\begin{equation} \label{Eq: T_iQ}
\begin{aligned} 
\xymatrix{
\qAg  \ar[d]_{T_i} \ar[rr]^{\Phi_\qQ} && \qK \ar[d]^{\bT^\qQ_i} \\ 
\qAg  \ar[rr]^{\Phi_\qQ} && \qK
}
\end{aligned}
\end{equation}

\subsection{Extended crystals} \label{Sec: extended crystal} 
In this subsection, we briefly review extended crystals. 
For the notion of classical crystals of $U_\qq(\g)$, we refer the reader to \cite{K91, K93, K95} and \cite[Chapter 4]{HK02}. 

\begin{definition}
	A \emph{crystal} is a set $B$ endowed with  maps $\wt\col  B \rightarrow \wlP$,  $\varphi_i$,  $\ep_i \col  B \rightarrow \Z \,\sqcup \{ -\infty \}$
	and $\te_i$, $\tf_i \col  B \rightarrow B\,\sqcup\{0\}$ 
 for all $i\in I$ which satisfy the following axioms: 
	\bna
	\item $\varphi_i(b) = \ep_i(b) + \langle h_i, \wt(b) \rangle$,
	\item $\wt(\te_i b) = \wt(b) + \alpha_i$ if $\te_i b\in B$, and 
$\wt(\tf_i b) = \wt(b) - \alpha_i$ if $\tf_i b\in B$, 
	\item for $b,b' \in B$ and $i\in I$, $b' = \te_i b$ if and only if $b = \tf_i b'$,
	\item for $b \in B$, if $\varphi_i(b) = -\infty$, then $\te_i b = \tf_i b = 0$,
	\item if $b\in B$ and $\te_i b \in B$, then $\ep_i(\te_i b) = \ep_i(b) - 1$ and $ \varphi_i(\te_i b) = \varphi_i(b) + 1$,
	\item if $b\in B$ and $\tf_i b \in B$, then $\ep_i(\tf_i b) = \ep_i(b) + 1$ and $ \varphi_i(\tf_i b) = \varphi_i(b) - 1$.
\ee
\end{definition}
Let $B(\infty)$ be the \emph{infinite crystal} of the negative half $U_{\qq}^-(\g_0)$. The $\Q(\q)$-algebra anti-involution $*: U_\qq(\g) \to  U_\qq(\g)$ induces the involution $*$ of $B(\infty)$. This involution $*$ gives another crystal functions $\ep_i^*$ and $\varphi_i^*$ and crystal operators
$$
\tf_i^* := * \circ \tf_i \circ *  \quad \text{ and } \quad \te_i^* := * \circ \te_i \circ *. 
$$
For any $i\in I$, we set 
\begin{align*}
	\iB := \{b\in \B \mid \ep_i(b) = 0  \} \quad \text{and} \quad \Bi := \{b\in \B \mid \eps_i(b) = 0  \}. 
\end{align*}
The \emph{Saito crystal reflections} (\cite{Saito94}) on the crystal $B(\infty)$ are defined as follows 
\begin{align*}
\cT_i(b) &:= \tf_i ^{ \vphi_i^*(b)} \te_i^{* \hskip 0.1em  \ep_i^*(b)}(b) \qquad \text{ for any } b\in \iB. 
\end{align*} 
The reflection $\cT_i$ is the crystal counterpart of \emph{Lusztig's braid symmetry} $T_{i,-1}'$ (\cite{L93}). 
We set $ \cTs_i = * \circ \cT_i \circ * $. 
Note that  $ \cTs_i $ is the inverse of $ \cT_i $.
Define  
\begin{align} \label{Eq: tSi}
\tcT_i	:= \cT_i ( \te_i^{\ep_i(b)} (b)  ).
\end{align}

\begin{example} \label{Ex: B(inf)}  \
\bni
\item 
A \emph{segment} is an interval $[a,b]$ for $1 \le a \le b \le n$, and a \emph{multisegment} is a multiset of segments. We denote by $\MS_n$ the set of all multisegments. We simply write $[a]$ for $[a,a]$.
For a multisegment $\bfm = \{ m_1, m_2, \ldots, m_k \} $, we write $\bfm  = m_1+m_2 + \cdots + m_k$. 
We denote the empty set $\emptyset$ by $1$.
It is known that $\MS_n$ has a crystal structure, which is isomorphic to $B(\infty)$ (see \cite{V01} for example). 
We refer the reader to the crystal description of $\MS_n$ in \cite[Section 7.1]{KP22}.
This realization is called the \emph{multisegment realization} of the infinite crystal $B(\infty)$.

\item Using the crystal description of $\MS_n$ in \cite[Section 7.1]{KP22},
it is easy to see that $\ep_i([a,b]) = \delta_{i,a}$, $\eps_i([a,b]) = \delta_{i,b}$ and 
\begin{align*}
\tf_i([a,b]) &= 
\begin{cases}
[a-1, b] & \text{if $i=a-1$,} \\ 
[i]+[a,b] & \text{otherwise,} 
\end{cases} \qquad
\te_i([a,b]) = 
\begin{cases}
1 & \text{if $i=a=b$,} \\ 
[a+1,b] & \text{if $i=a$ and $a<b$,} \\ 
0 & \text{otherwise.} 
\end{cases} \\ 
\tfs_i([a,b]) &= 
\begin{cases}
[a, b+1] & \text{if $i=b+1$,} \\ 
[i]+[a,b] & \text{otherwise,} 
\end{cases} \qquad
\tes_i([a,b]) = 
\begin{cases}
1 & \text{if $i=a=b$,} \\ 
[a,b-1] & \text{if $i=b$ and $a<b$,} \\ 
0 & \text{otherwise.} 
\end{cases}
\end{align*}
\ee

\end{example}

The \emph{extended crystal} of $B(\infty)$ is defined as 
\begin{align} \label{Eq: extended crystal}
	\cB(\infty) \seteq   \Bigl\{  (b_k)_{k\in \Z } \in \prod_{k\in \Z} B(\infty)  \bigm | b_k =\mathsf{1} \text{ for all but finitely many $k$}  \Bigr\},
\end{align}
where $\mathsf{1}$ is the highest weight element of $B(\infty)$.
We set $\one := (\mathsf{1})_{k\in \Z} \in \cB(\infty) $.
For any $\cb = (b_k)_{k\in \Z} \in \cB(\infty)$, $i\in I$ and $k\in \Z$, 
we define the \emph{extended crystal operator} $ \tF_{(i,k)} \col  \cB(\infty) \longrightarrow  \cB(\infty)$ on $\cB(\infty)$ 
by
\begin{equation*} 
\begin{aligned}
	\tF_{(i,k)}(\cb) & \seteq  
	\begin{cases}
		(\cdots , b_{k+2},  \  b_{k+1} , \ \tf_i( b_k), \ b_{k-1}, \cdots ) & \text{ if } \hep_{( i,k)} (\cb) \ge 0,\\
		(\cdots , b_{k+2},  \ \tes_i (b_{k+1} ), \ b_k, \ b_{k-1}, \cdots ) & \text{ if }  \hep_{( i,k)} (\cb) < 0 ,
	\end{cases}
\end{aligned}
\end{equation*}
where $ \hep_{(i,k)}(\cb) \seteq   \ep_{i}(b_k) - \eps_{i}(b_{k+1})$ (see \cite{KP22, P23} for details).  

For $i\in I$ and $\cb = (b_k)_{k\in \Z} \in \cB(\infty)$, we define 
\begin{align} \label{Eq: def of Ri}
R_i(\cb) = (b_k')_{k\in \Z}
\end{align}
by
\begin{align*}
b_k' := \tfss{ \ep_i(b_{k-1}) }_i  \left( \tcT_{i}(b_k) \right)
\end{align*}	
for any $k\in \Z$. Hence we have the map 
$R_i : \cB(\infty) \longrightarrow \cB(\infty). $

\begin{theorem} [{\cite[Lemma 3.2 and Theorem 3.4]{P23}}]  \
\bni
\item The maps $R_i$ are bijective.
\item The maps $R_i$ satisfy the braid  group relations for $ \bg_n$.
\ee
\end{theorem}	
Note that the inverse of $R_i$ is given as $ R_i^* =  \star \circ R_i \circ \star $ (see \cite[Remark 3.3]{P23}).

\begin{example}  \label{Ex: cB(inf)} \

\bni
\item Define
$$
\hMS_n := \{ (\bfm_k)_{k\in \Z} \in \prod_{k\in \Z} \MS_n \mid \bfm_k=\emptyset \text{ for all but finitely many $k$} \}.
$$
Using the multisegment realization $\MS_n$ in Example \ref{Ex: B(inf)}, one can show that $\hMS_n$ has the extended crystal structure and $\hMS_n$ is isomorphic to $\cB(\infty)$. 

\item 
For any $b\in B(\infty)$ and $k\in \Z$, we define 
\begin{align*}
	\ie_k(b) = (b_t')_{t\in \Z} \in \cB(\infty),
\end{align*}	
where 
$
b_t'=
\begin{cases}
	b_k & \text{ if } t = k,\\
	1 & \text{ otherwise.}
\end{cases}
$
For a segment $[a,b]$ and $k\in \Z$, we set 
$$
[a,b]_k :=  \ie_k([a,b]).
$$
Using the crystal structure of $\hMS_n$ with the definition of $R_1$, we have 
$$ 
R_1  ([a,b]_k) = 
\begin{cases}
	[1]_{k+1} & \text{ if $a=b = 1$}, \\
 	[1]_{k+1}+[1,b]_k & \text{ if $a = 1$ and $b>1$}, \\
  	[1,b]_k  & \text{ if $a = 2$},  \\
	[a,b]_k  & \text{ otherwise}.
\end{cases}
$$	
\ee
	
\end{example}	

\subsection{Categorical crystals} \label{Sec: cat crystal} 
In the subsection, we briefly review the categorical crystal of $\catCO$. 
Recall that $U_q'(\g)$ is of affine type $A_n^{(1)}$.

Let $\ddD_\qQ := \{ \Rt_i\}_{i\in I_0}$ be the complete duality datum arising from a $Q$-datum $\qQ$.
Let $\sB(\g)$ be the set of the isomorphism classes of simple modules in $\catCO$.  
It is proved in \cite{KP22} that $\sB(\g)$ has the categorical crystal structure arising from $\ddD_\qQ$, which is defined by
$$
\ttF_{i,k} (M) \seteq   \hd \left( (\dual^k \Rt_i) \tens M) \right),  \qquad  \ttE_{i,k} (M) \seteq    \hd \left(M \tens (\dual^{k+1} \Rt_i) \right)
$$
for $ M \in \sB(\g)$, $ i\in I_0$, $k\in \Z$ and that $\sB(\g)$ is isomorphic to the extended crystal $\cB(\infty)$ of finite type $A_n$. Here, $\hd(U)$ denotes the \emph{head} of $U$.

 For a simple  module $V \in \catCO$, we write $[V] \in \sB(\g)$ for the isomorphism class of $V$. We sometimes regard $\sB(\g)$ as a subset of $\qK$ via the $(q,t)$-characters.
The extended crystal isomorphism is given as
\begin{align} \label{Eq: phiQ}
\phi_\qQ : \cB(\infty) \buildrel \sim \over \longrightarrow \sB(\g)
\end{align}
sending $ \tF_{i,k} (\one)$ to the equivalence class $\ttF_{i,k} (\one) =  [ \dual^k( \Rt_i)]$ for any $i\in I_0$ and $k\in \Z$.
Then, for each $i\in I_0$, the braid group actions $\bR_i^\qQ$ on $\sB(\g)$ is defined by the following commutative diagram:
\begin{equation} \label{Eq: R_iQ}
\begin{aligned} 
\xymatrix{
\cB(\infty)  \ar[d]_{R_i} \ar[rr]^{\phi_\qQ} && \sB(\g) \ar[d]^{\bR^\qQ_i} \\ 
\cB(\infty)  \ar[rr]^{\phi_\qQ} && \sB(\g)
}
\end{aligned}
\end{equation}

\vskip 2em

\section{Explicit descriptions for type \texorpdfstring{$A_n$}{An}} \label{Sec: ed An}

In this section, we choose a special $Q$-datum and investigate relevant notations and properties.
This choice of the $Q$-datum will be used to compare with the braid group actions in Grassmannian.

Let $U_q'(\g)$ be the quantum affine algebra of type $A_{n}^{(1)}$. We set 
\begin{align} \label{Eq: An Q-datum}
\qQ := (\Dynkin, \xi),
\end{align}
where $\Dynkin$ is the Dynkin diagram of finite type $A_{n}$ and $\xi: I_0 \rightarrow \Z$ is the height function defined by $\xi(i) \seteq  i-1 $ for $i\in I_0$. 
Note that the $Q$-datum $\qQ$ is the same as in \cite[Section 7.2]{KP22}. 
The sets $\crhI_\Z$ and $\crhI_\qQ$ defined in Section \ref{Subsection: HL cat} are given as follows:
\begin{equation} \label{Eq: S S+ SQ}
\begin{aligned}
\crhI_\Z &\seteq  \{ (i, a) \in I_0 \times \Z \mid   a-i \equiv 1 \bmod 2  \}, \\ 
\crhI_{\Z_{\ge0}} &\seteq  \{ (i, a) \in \crhI_\Z \mid   a \ge i-1  \}, \\ 
\crhI_\qQ & \seteq  \{ (i, a) \in \crhI_\Z  \mid  i-1 \le a \le 2n-i-1    \}.
\end{aligned}
\end{equation}
For any $(i,a) \in \crhI_\Z$,  we define 
\begin{align*}
\cdual(i,a) \seteq  (n+1 - i, a+n+1) \quad \text{ and } \quad  \cdual^{-1}(i,a) \seteq  (n+1 - i, a-n-1),
\end{align*}
and set $ \cdual (A)\seteq \{ \cdual (i,a) \mid (i,a) \in A  \}$ for a subset $A \subset \crhI_\Z$.
Note that 
$$
\crhI_{\Z_{\ge0}} = \bigcup_{k\in \Z_{\ge0}} \cdual^k (\crhI_\qQ) \quad \text{ and } \quad  \crhI_{\Z} = \bigcup_{k\in \Z} \cdual^k (\crhI_\qQ). 
$$

The corresponding Hernandez-Leclerc category $\catCO$ (resp.\ $\catCP$, $\catCQ$) is determined by the fundamental modules $ V(Y_{i,a}) $ for all $(i,a) \in \crhI_\Z$ (resp.\ $ \crhI_{\Z_{\ge0}}$, $ \crhI_\qQ$). Note that 
\begin{align} \label{Eq: dual to cdual} 
[\dual (V(Y_{i,a}))] = [V(Y_{\cdual(i,a)})]
\end{align}
(see \eqref{Eq: dual V}).

\begin{remark}
We keep the notations appearing in Remark \ref{Rmk: V(varpi)}.
We set $o(1)=-1$ and $o(i) = -o(i-1)$ for $i\in [2,n]$. Under the choice of the $Q$-datum $\qQ$ given in \eqref{Eq: An Q-datum}, 
the module $ V(Y_{i,a})$ coincides with $V(\varpi_i)_{(-q)^a}$ by shifting $(-q)^{n+1}$.
\end{remark}

\begin{example} \label{Ex: pic for crhI}
Let $n=3$. The set $\crhI_{\Z_{\ge0}}$ can be drawn pictorially as follows.
$$ 
\scalebox{0.8}{\xymatrix@C=1ex@R=  0.0ex{ 
i \diagdown\, a     \ar@{-}[dddd]<3ex> \ar@{-}[rrrrrrrrrrrrrrrr]<-2.0ex>    & -1 & 0 & 1 & 2 & 3 & 4 & 5 & 6 & 7 & 8 &9&10 & 11& 12 & & &   \\
1       & &  \ast & & \ast  &  & \ast  & &  \bullet & & \trg && \trg && \trg   \\
2     &  & & \ast & & \ast &  &  \bullet  & &  \bullet  && \trg && \trg  \\
3      & &  & & \ast & &  \bullet  & &  \bullet && \bullet && \trg && \cdot \\
& 
}}
$$	
The elements of $\crhI_\qQ$ (resp.\ $\cdual(\crhI_\qQ$), $\cdual^2(\crhI_\qQ$)) are marked in $\ast$ (resp.\ $\bullet$, $\trg$). 
Note that the fundamental modules $V(Y_{i,a})$ in $\catCQ$ (resp.\ $ \dual (\catCQ)$, $\dual^2 (\catCQ)$) are indexed by $(i,a)$ placed at $\ast$ (resp.\ $\bullet$, $\trg$). 
\end{example}

The complete duality datum  $\ddD$ arising from the $Q$-datum $\qQ$ is given as 
\begin{align} \label{Eq: dd from Q}
\ddD = \{ \Rt_i\}_{i\in I_0},
\qt{where $\Rt_i \seteq   V(Y_{1, 2i-2}) $ for $i\in I_0$.}
\end{align}
We simply write $ \Phi := \Phi_\qQ$, where $ \Phi_\qQ $ is given in \eqref{Eq: Phi_Q}. 
Note that 
\begin{align} \label{Eq: Phi fim}
\Phi(f_{i,m}) =  [ \dual^m( \Rt_i) ] = [V(Y_{ \cdual^m(1, 2i-2)})] \qquad \text{ for any $i\in I_0$ and $m\in \Z$.} 
\end{align}

\subsection{Multisegment realizations}
We choose the reduced expression 
$$
\rxw_0 = (s_1) (s_2s_1) (s_3s_2s_1) \cdots (s_n \cdots s_2s_1)
$$
of the longest element $w_0$ in the symmetric group $\mathfrak{S}_{n+1}$, and set $\alpha_{a,b} := \al_a + \al_{a+1} + \cdots + \al_{b}$ for $1 \le a \le b\le n$. 
Then the corresponding \emph{dual PBW vector} $ E^*(\al_{a,b})$ belongs to the \emph{upper global basis} (or \emph{dual canonical basis}) and matches with the segment $[a,b]$ in the infinite crystal $B(\infty)\simeq \MS_n$ (see \cite[Section 7.2]{KP22}).

Recall the multisegment realizations $\hMS_n$ of $\cB(\infty)$ given in Example \ref{Ex: cB(inf)}.
We identify $\hMS_n$ with $\cB(\infty)$.

Let  $\sB(\g)$ be  the categorical crystal  associated with the $Q$-datum $\qQ$ (see Section \ref{Sec: cat crystal}). 
We simply write $\bR_i := \bR_i^\qQ$, where $\bR_i^\qQ$ is the Braid group action on $\sB(\g)$  given in \eqref{Eq: R_iQ}.
The complete duality datum \eqref{Eq: dd from Q} induces an isomorphism 
\begin{align} \label{Eq: [a,b]}
\phi:= \phi_\qQ: \hMS_n \buildrel \sim \over \longrightarrow \sB(\g), \qquad 	[a,b]_k \mapsto [V(Y_{\pk_{[a,b]_k}})],
\end{align}
where $\pk_{[a,b]_k} := \cdual^k (b-a+1, a+b-2)$ (see \eqref{Eq: phiQ}).
Then we have 
$ \phi(\bfm) = [V(\bfm)]$
where 
\begin{align} \label{Eq: Vm}
[V(\bfm)] := \left[V \left( \prod_{s}  Y_{\pk_{[a_s, b_s]_{k_s}}}^{c_s} \right) \right] \qquad \text{for any $ \bfm =  \sum_{s} c_s [a_s, b_s]_{k_s} \in \hMS_n  $}
\end{align}
by \cite[Theorem 6.1]{KKOP23} (see also \cite[Theorem 5.10, Section 7.2]{KP22}).

\begin{example} \label{Ex: pic for [a,b]}
We continue Example \ref{Ex: pic for crhI}.
Pictorially, the elements $[a,b]_k$ are drawn at the position $ \pk_{[a,b]_k}$ as follows:
$$ 
\scalebox{0.7}{\xymatrix@C=1ex@R=  0.0ex{ 
		i \diagdown\, a     \ar@{-}[dddd]<3ex> \ar@{-}[rrrrrrrrrrrrrrrr]<-2.0ex>    & -1 & 0 & 1 & 2 & 3 & 4 & 5 & 6 & 7 & 8 &9&10 & 11& 12 & & &   \\
		1       & &  [1]_0 & & [2]_0  &  & [3]_0  & &  [1,3]_1 & & [1]_2 && [2]_2 && [3]_2   \\
		2     &  & & [1,2]_0 & & [2,3]_0 &  &  [1,2]_1  & & [2,3]_1  && [1,2]_2 && [2,3]_2  \\
		3      & &  & & [1,3]_0 & &  [1]_1  & &  [2]_1 && [3]_1 && [1,3]_2 && \cdot \\
		& 
}}
$$	
\end{example}

\subsection{Braid group actions}
Recall the braid group actions $T_i$ and $\bT_i^\qQ$ defined in \eqref{Eq: T_iQ}.
We simply write $\bT_i := \bT_i^\qQ$.

For any $k\in \Z$, we define the \emph{spectral parameter shift}
\begin{align} \label{Eq: sps}
\shift_k : \qK \buildrel \sim \over \longrightarrow  \qK
\end{align}
by $ \shift_k ( Y_{i,a}) = Y_{i,a + 2k} $ for any $(i,a) \in \crhI_\Z$. It is easy to see that 
\begin{align} \label{Eq: Sh T= T Sh}
   \bT_i \circ \shift_{n+1}  = \shift_{n+1} \circ \bT_i \qquad \text{ for any $i\in I_0$.}
\end{align}

\begin{lemma} \label{Lem: Ti respects B}
Let $ C = T_1 T_2 \cdots T_n$.
\bni
\item  For any $k\in I_0$, we have $ T_k = C^{k-1} T_1 C^{1-k} $.
\item We have $ \Phi \circ C = \shift_1 \circ \Phi $, i.e., 
$$
\xymatrix{
	\qAg  \ar[d]_{C} \ar[rr]^{\Phi} && \qK  \ar[d]^{\shift_1} \\
	\qAg  \ar[rr]^{\Phi } && \qK. 
}
$$
In particular, $ \bT_k = \shift_{k-1} \circ \bT_1 \circ \shift_{1-k} $ for $k\in I_0$.
\item \label{Lem: Ti respects B (iii)}
For any $i\in I_0$ and $[V] \in \sB(\g)$,  we have 
$$\bT_i ([V])  \in  \sB(\g),$$
where we regard $\sB(\g)$ as a subset of $\qK$.
\ee
\end{lemma}	
\begin{proof}
(i) It follows from a direct computation.

(ii) Let $m\in \Z$ and $i\in I_0$.

Suppose that $i \ne n$. It is easy to see that  $T_{i} T_{i+1} (f_{i,m}) = f_{i+1, m}$ (see \cite[Example 5.10]{OP24} for example).	By \eqref{Eq: Phi fim}, we have
$$
\Phi\circ C(f_{i,m}) = \Phi (f_{i+1,m}) = [V(Y_{ \cdual^m(1, 2i)})] =  \shift_1( [V(Y_{ \cdual^m(1, 2i-2)})]) = \shift_1\circ \Phi(f_{i,m}).
$$

Suppose that $i=n$. 
Note that 
$$ (1-\qq^2)E^*(\al_{1,k+1}) = E^*(\al_{1,k}) E^*(\al_{k+1}) - \qq E^*(\al_{k+1} ) E^*(\al_{1,k}). 
$$
(see \cite[Theorem 4.7]{BKM14} for example).
Let $M_k := T_1 T_2 \cdots T_{k-1}(f_{k,0})$. Since $M_k$ corresponds to $[1,k]_0$ in the extended crystal, the isomorphism \eqref{Eq: [a,b]} tells us that 
$$ 
\Phi(M_k) =  \Phi( T_1T_2 \cdots T_{k-1} (f_{k,0}) ) = [V(Y_{k, k-1})]. 
$$
Thus, it follows from \eqref{Eq: dual to cdual} and \cite[Lemma 4.1 (i)]{OP24} that 
\begin{align*}
\Phi\circ C(f_{n,m}) &= \Phi (T_1T_2 \cdots T_{n-1} (f_{n, m+1})) =   \dual^{m+1} [V(Y_{n, n-1})] \\
& = [V(Y_{ \cdual^{m+1} ( n, n-1)})] = [V(Y_{ \cdual^{m} ( 1, 2n)})] = \shift_1( [V(Y_{ \cdual^{m} ( 1, 2n-2)})]) \\
&= \shift_1\circ \Phi(f_{n,m}).
\end{align*}

(iii) Let $[V] \in \sB(\g)$. It is obvious that $ \shift_1([V]) \in  \sB(\g) $. Since $1$ is a sink in the $Q$-datum $\qQ$, we have  $ \bT_1([V]) \in  \sB(\g) $ by \cite[Proposition 6.2]{FHOO23}. Thus, (i) and (ii) give the assertion.
\end{proof}	

\begin{remark}
Lemma \ref{Lem: Ti respects B} (iii) can be induced from \cite[Theorem 7.3]{KKOP24A} in terms of global basis of $\qAg$. 
\end{remark}

\begin{definition} \
\bna
\item An element $b\in \sB(\g)$  is called \emph{real} if $b^2$ is contained in $\sB(\g)$. 
\item An element $b\in \sB(\g)$ is called \emph{imaginary} if $b$ is not real.
\item An element $b\in \sB(\g)$ is called \emph{prime} if $b$ cannot be expressed as a product of elements of $\sB(\g)$.
\ee
\end{definition}

We then have the following by Lemma \ref{Lem: Ti respects B} (iii). 

\begin{corollary} \label{Cor: real}
Let $b\in \sB(\g)$ and $i\in I$.
\bni
\item If $b$ is real then so is $\bT_i(b)$.
\item If $b$ is imaginary then so is $\bT_i(b)$.
\item If $b$ is prime then so is $\bT_i(b)$.
\ee
\end{corollary}

\vskip 2em 

\section{Infinite Grassmannian cluster algebras} \label{Sec: igca}

In this section, we define braid group actions on the infinite Grassmannian cluster algebra by extending the braid group actions introduced by Fraser \cite{Fra20} and investigate several properties.  

We fix $m \in \Z_{>1}$ throughout this section.
Define 
\begin{align*}
\cvC{m}&:= \{  (i_1,i_2, \ldots, i_m) \in \Z^{ m} \mid i_1 < i_2 < \ldots < i_m \}, \\
\cvF{m}&:= \{ (i, i+1 \ldots, i+m-1 ) \in \Z^{m} \mid i\in \Z \}. 
\end{align*}
For any $\bfi = (i_1,i_2, \ldots i_m)\in \Z^{ m}$,  $t \in \Z$ and $a \in \Z\setminus\{0\}$,
we define 
$$
\bfi +t := (i_1 + t, i_2 + t, \ldots, i_m + t )\quad  \text{and} \quad a\bfi := (a i_1, a i_2 , \ldots, a i_m ).
$$  
Note that $\cvC{m}$ and $\cvF{m}$ are closed under addition by $t\in \Z$. For an interval $[k,l] \subset \Z$ and $\bfi = (i_1, i_2, \ldots, i_m)\in \cvC{m}$, we simply write $ \bfi \subset [k,l] $ if $ \{ i_1, i_2, \ldots, i_m \} \subset [k,l] $. 

For each $\bfi = (i_1, i_2, \ldots, i_m) \in \cvC{m}$, we set $ \vP_\bfi = \vP_{i_1, i_2, \ldots, i_m} $ to be an indeterminate, and define $ \vP_{i_{w(1)}, i_{w(2)}, \ldots, i_{w(m)}  } := (-1)^{\ell(w)} \vP_{i_1,i_2, \ldots, i_m}$ for any $w \in \sg_{m}$. 
For any $\bfj = (j_1, j_2, \ldots, j_m) \in \Z^{ m}$, we understand $ \vP_{j_1, j_2, \ldots, j_m} = 0$ if there exist $ s,t \in [1,m]$ such that $s \ne t$ and $j_s = j_t$.  

We define $\widetilde{A} := \C[\vP_\bfi, \ \vP_\bfj^{-1} \mid \bfi \in \cvC{m},\ \bfj \in \cvF{m}]$ to be the algebra obtained from the polynomial algebra $\C[\vP_\bfi \mid \bfi \in \cvC{m}]$ by localizing the indeterminates $\vP_\bfj$ for any $\bfj \in \cvF{m}$. 
We define $\iGr_m$ to be the quotient algebra of $\widetilde{A}$ by the ideal generated by
\begin{align} \label{Eq: Plucker}
\sum_{t=1}^{m+1} (-1)^t \vP_{i_1, \ldots, i_{m-1}, j_t} \vP_{j_1, \ldots, \widehat{j_{t}}, \ldots,  j_{m+1}}
\end{align}
for any $\bfi = (i_1,i_2, \ldots, i_{m-1} ) \in \cvC{m-1}$ and $\bfj = (j_1,j_2, \ldots, j_{m+1} ) \in \cvC{m+1}$.
Denote by $\biGr_m$  the quotient of $\iGr_m$ by the ideal generated by 
$ \vP_{\bfi} - 1 $ for all $\bfi \in \cvF{m}$.
For an interval $[k,l]$, we define $\iGr_m^{[k,l]}$ (resp.\ $\biGr_m^{[k,l]}$) to be the subalgebra of $\iGr_m$ (resp.\ $\biGr_m$) generated by $\vP_\bfi$ for all $\bfi \in \cvC{m}$  with $\bfi \subset [k,l]$. We set 
\begin{align*}
\iGr_m^{N} := \iGr_m^{[1,N]}, \qquad   \iGr_m^+ := \iGr_m^{[1, \infty]},  \qquad \biGr_m^{N} := \biGr_m^{[1,N]} \quad \text{and} \quad   \biGr_m^+ := \biGr_m^{[1, \infty]}
\end{align*}
for any $N \in \Z_{>0}$.  
By the definition, for each $t\in \Z$, we have the $\C$-automorphism  
\begin{align} \label{Eq: auto sh}
\sh_t:\iGr_m \buildrel \sim \over \longrightarrow \iGr_m 
\end{align}
defined by 
$ \vP_{\bfi} \mapsto \vP_{\bfi+t} $ for any $\bfi \in \cvC{m}$. 
Note that $\sh_t$ is well-defined on the quotient $\biGr_m$.

\subsection{Grassmannians}  
Denote by $\tGr(m,N)$ the \emph{affine cone} over the Grassmannian of $m$-dimensional subspaces in $\C^N$ (see \cite[Section 4.1.5]{LR08} and see also \cite[Section 3]{Fra20}). 
We define $\tdGr(m,N) $ to be the Zariski-open subset of $\tGr(m,N)$ cut out by the vanishing of the \emph{Pl\"{u}cker coordinates} $\Delta_{\bfi}$ for any $\bfi \in \cvF{m}$ with $ \bfi \subset [1,N] $. Similarly, we also define $\toGr(m,N)$ to be the Zariski-open subset of $\tGr(m,N)$ cut out by the vanishing of $\Delta_{\bfi}$ for any  $\bfi = (i,i+1,\ldots, i+m-1) \pmod N$ with $ \bfi \subset [1,N] $.
Note that $\toGr(m,N)$ is the same as that in \cite[Section 3]{Fra20}. 
The coordinate ring $\C[\tdGr(m,N)]$ (resp.\ $\C[\toGr(m,N)]$) is the localization of the coordinate ring $\C[\tGr(m,N)]$ at the elements $\Delta_{\bfi}$ for any $\bfi \in \cvF{m}$ (resp.\ $\bfi = (i,i+1,\ldots, i+m-1) \pmod N$) with $ \bfi \subset [1,N] $. Note that $\C[\tdGr(m,N)]$ can be regarded as a subalgebra of $\C[\toGr(m,N)]$.

\begin{lemma} \label{Lem: phimN}
For any $ m, N \in \Z_{> 0} $ with $m < N$, there is an algebra isomorphism
$$
\psi_{m, N} : \C[\tdGr(m, N)]  \buildrel \sim \over \longrightarrow \iGr_m^{N}
$$
mapping $ \Delta_\bfi$ to $\vP_\bfi$ for any $\bfi \in \cvF{m}$ with $\bfi \in [1,N]$.
\end{lemma}

\begin{proof}
The result follows from Section 4.1.5 in \cite{LR08}.
\end{proof}

A cluster algebra structure of the algebra $\C[ \tGr(m,N)]$ was introduced in \cite{Sco06}. Since the set $ \{ \Delta_\bfi \mid \bfi \in \cvF{m} \}$ is contained in the set of frozen variables of the cluster algebra $\C[ \tGr(m,N)]$, the algebra $ \C[\tdGr(m, N)]$ has the same cluster algebra structure of $\C[ \tGr(m,N)]$ except that the frozen variables are invertible. 

\begin{remark} \label{Rmk: Sigma seed}
Denote by $\Sigma_{m,N}$ the seed of $\C[ \tdGr(m,N)]$ given by the initial quiver $Q^N_m$ and initial cluster variables $\vP_\bfi$ defined in the following, \cite{Sco06}. The quiver $Q_m^N$ has vertices $(0, 0)$ and $(a,b)$ for $a \in [1,N-m]$ and $b \in [1,m]$, and arrows are defined as follows:  
\bna
\item $(1,1) \to (0,0)$,
\item $(a,b) \to (a-1,b)$ for $a \in [2,N]$ and $b \in [1,m]$, 
\item $(a, b) \to (a,b-1)$ for $a \in [1,N-m]$ and $b \in [2, m]$,
\item  $(a,b) \to (a+1,b+1)$ for $a \in [1,N-m-1]$ and $b \in [1,m-1]$.
\ee
 The frozen variable at $(0,0)$ is $\vP_{1,\ldots,m}$. The variables in column $b$ ($b \in [1,m-1]$) are $\vP_{1,2,\ldots, m-b, m-b+2, \ldots, m+1}$, $\vP_{1,2,\ldots, m-b, m-b+3, \ldots, m+2}$, $\ldots$, $\vP_{1,2,\ldots, m-b, N-b+1, \ldots, N}$, where the last one is a frozen variable and others are cluster variables. Column $m$ consists of frozen variables $\vP_{2,\ldots, m+1}$, $\vP_{3,\ldots, m+2}$, $\ldots$, $\vP_{N-m+1, \ldots, N}$. 
 \end{remark}

We recall the definition of \emph{rooted cluster algebras} introduced by Assem, Dupont, and Schiffler in \cite[Definition 1.4]{ADS14} which will be used in the proof Lemma \ref{lem:Grplus has a cluster algebra structure induced from GrmN}. Let $\Sigma=({\bf x}, {\bf ex}, B)$ be a seed, where ${\bf x}$ is a cluster, ${\bf ex}$ is the set of exchange variables, $B$ is the exchange matrix. A sequence $(x_1, \ldots, x_l)$ is called $\Sigma$-admissible if $x_1$ is exchangeable in $\Sigma$, and for every $i \ge 2$, $x_i$ is exchangeable in $\mu_{x_{i-1}}\circ \cdots \circ \mu_{x_1}(\Sigma)$. Denote
\begin{align*}
    {\rm Mut}(\Sigma) = \{ \mu_{x_l} \circ \cdots \circ \mu_{x_1}(\Sigma) : l > 0, (x_1, \ldots, x_l) \text{ is $\Sigma$-admissible} \}.
\end{align*}
The rooted cluster algebra with a fixed initial seed $\Sigma$ is the pair $(\Sigma, \mathcal{A})$, where $\mathcal{A}=\mathcal{A}(\Sigma)$ is the $\Z$-subalgebra of the ambient field $\mathcal{F}_{\Sigma} = \mathbb{Q}(x \mid x \in {\bf x})$ given by
\begin{align*}
\mathcal{A} = \Z[ x \mid x \in \bigcup\nolimits_{({\bf x}, {\bf ex}, B) \in {\rm Mut}(\Sigma)} {\bf x} ]. 
\end{align*}
We also recall the concept of rooted cluster morphism \cite[Definition 2.2]{ADS14}. For two seeds $\Sigma=({\bf x}, {\bf ex}, B)$, $\Sigma'=({\bf x}', {\bf ex}', B')$, and a map $f: \mathcal{F}_{\Sigma} \to \mathcal{F}_{\Sigma'}$, a sequence $(x_1, \ldots, x_l)$ is called $(f, \Sigma, \Sigma')$-biadmissible if it is $\Sigma$-admissible and $(f(x_1), \ldots, f(x_l))$ is $\Sigma'$-admissible. A rooted cluster morphism $f: \mathcal{A}(\Sigma) \to \mathcal{A}(\Sigma')$ is a ring homomorphism such that 
\begin{itemize}
\item $f(\bf x) \subset {\bf x}' \sqcup \Z$,
\item $f({\bf ex}) \subset {\bf ex}' \sqcup \Z$,
\item and for every $(f, \Sigma, \Sigma')$-biadmissible sequence $(x_1, \ldots, x_l)$ and for any $y \in {\bf x}$, the following holds
\[f( \mu_{x_l} \circ \cdots \circ \mu_{x_1}(y) ) = \mu_{f(x_l)} \circ \cdots \circ \mu_{f(x_1)}(f(y)).\]
\end{itemize}

By Proposition 5.1 and Theorem 5.2 in \cite{GG18}, a quantum version of $\iGr_k^+$ is the colimit of the quantum cluster algebra structures
on $\C_q[ \tdGr(m,N)]$. Similarly, we have the following result. 

\begin{lemma} \label{lem:Grplus has a cluster algebra structure induced from GrmN}
The algebra $\iGr_m^+$ has a cluster algebra structure induced from the cluster algebras $\C[ \tdGr(m,N)]$.
\end{lemma}
\begin{proof}
The proof is similar to the proof of Proposition 5.1 and Theorem 5.2 in \cite{GG18}. The seed $\Sigma_{m,N}$ is a subseed of $\Sigma_{m,N+1}$. Denote by $\iota_N: \C[ \tdGr(m,N)] \to \C[ \tdGr(m,N+1)]$ given by $\iota_k(\vP_J) = \vP_J$. The morphism $\iota_N$ is a rooted cluster morphism. The family $\{\iota_N: N \ge 1 \}$ has a colimit and the colimit is $\iGr_m^+$. 
\end{proof}

\begin{figure} 
    \centering
    \adjustbox{scale=0.66,center}{%
    \begin{tikzcd}
	{\fbox{$\vP_{1234} \ (0,0)$}} \\
	{\vP_{1235} \ (1,1)} & {\vP_{1245} \ (1,2)} & {\vP_{1345} \ (1,3)} & {\fbox{$\vP_{2345} \ (1,4)$}} \\
	{\vP_{1236} \ (2,1)} & {\vP_{1256} \ (2,2)} & {\vP_{1456} \ (2,3)} & {\fbox{$\vP_{3456} \ (2,4)$}} \\
	{\vP_{1237} \ (3,1)} & {\vP_{1267} \ (3,2)} & {\vP_{1567} \ (3,3)} & {\fbox{$\vP_{4567} \ (3,4)$}} \\
	{\vP_{1238} \ (4,1)} & {\vP_{1278} \ (4,2)} & {\vP_{1678} \ (4,3)} & {\fbox{$\vP_{5678} \ (4,4)$}} \\
	\vdots & \vdots & \vdots & \vdots
	\arrow[from=2-1, to=1-1]
	\arrow[from=3-1, to=2-1]
	\arrow[from=4-1, to=3-1]
	\arrow[from=5-1, to=4-1]
	\arrow[from=2-2, to=2-1]
	\arrow[from=2-3, to=2-2]
	\arrow[from=2-4, to=2-3]
	\arrow[from=3-2, to=2-2]
	\arrow[from=4-2, to=3-2]
	\arrow[from=5-2, to=4-2]
	\arrow[from=5-3, to=4-3]
	\arrow[from=4-3, to=3-3]
	\arrow[from=3-3, to=2-3]
	\arrow[from=5-4, to=4-4]
	\arrow[from=4-4, to=3-4]
	\arrow[from=3-4, to=2-4]
	\arrow[from=3-4, to=3-3]
	\arrow[from=3-3, to=3-2]
	\arrow[from=3-2, to=3-1]
	\arrow[from=4-4, to=4-3]
	\arrow[from=4-3, to=4-2]
	\arrow[from=4-2, to=4-1]
	\arrow[from=5-4, to=5-3]
	\arrow[from=5-3, to=5-2]
	\arrow[from=5-2, to=5-1]
	\arrow[from=6-1, to=5-1]
	\arrow[from=5-1, to=6-2]
	\arrow[from=6-2, to=5-2]
	\arrow[from=5-2, to=6-3]
	\arrow[from=6-3, to=5-3]
	\arrow[from=5-3, to=6-4]
	\arrow[from=6-4, to=5-4]
	\arrow[from=2-1, to=3-2]
	\arrow[from=2-2, to=3-3]
	\arrow[from=2-3, to=3-4]
	\arrow[from=3-1, to=4-2]
	\arrow[from=3-2, to=4-3]
	\arrow[from=3-3, to=4-4]
	\arrow[from=4-1, to=5-2]
	\arrow[from=4-2, to=5-3]
	\arrow[from=4-3, to=5-4]
\end{tikzcd} }
    \caption{An initial seed for $\iGr_4^{+}$. The $(i,j)$ on vertices are positions of the vertices.}
    \label{fig:an initial quiver for iGr4plus}
\end{figure}

An initial seed for $\iGr_m^+$ is obtained from the initial seed of $\C[ \tdGr(m,N)]$ described in Remark \ref{Rmk: Sigma seed} by sending $N \to \infty$. An initial seed of $\iGr_4^{+}$ is given in Figure \ref{fig:an initial quiver for iGr4plus} as an example. Since $\vP_\bfi$ ($\bfi \in \cvF{m}$) are frozen variables in $\iGr_m^{+}$, the algebra $\biGr_m^+$ has the same cluster algebra structure of $\iGr_m^{+}$ except that there are no frozen variables.

\subsection{Braid group actions} \label{subsec: bg action}
Let $ \tsig_i^{(m,N)}$ be Fraser's braid group action on the localization of $\C[ \toGr(m,N)]$ (see \cite[Section 4]{Fra20}). If no confusion arises, then we write 
$$
\tsig_i := \tsig_i^{(m,N)}
$$
for notational simplicity.
When $i=1$, we recall the action $ \tsig_1$ briefly in terms of configurations of vectors following \cite{Fra20}. We assume that $m$ divides $N$ for the purpose of this paper. 
Let $V := \C^m$. An $N$-tuple $ \textbf{v} =  (v_1, v_2, \ldots, v_N) \in V^N$ is \emph{consecutively generic} if every cyclically consecutive $m$-vectors are linearly independent in $V$. We denote by $(V^N)^\circ \subset V^N$ the set of consecutively generic $N$-tuples, and identify points of $\toGr(m,N)$ with points in $\mathrm{SL}(V) \backslash (V^N)^\circ $ (see \cite[Section 3]{Fra20}).  
An $N$-tuple $(v_1,v_2,\ldots, v_N) \subset \mathrm{SL}(V) \backslash (V^N)^\circ$ is divided into \emph{windows} $[v_1, \ldots, v_m]$, $[v_{m+1},\ldots, v_{2m}]$, $\ldots$, $[v_{N-m+1},\ldots, v_{N}]$. 
The map 
$$
\kappa:\mathrm{SL}(V) \backslash (V^N)^\circ \longrightarrow \mathrm{SL}(V) \backslash (V^N)^\circ 
$$
is defined by the following: the first two columns $ v_{tm+1}$, $ v_{tm+2}$ are changed to $ v_{tm+2}$, $ w_{tm+1}$ in each window, where $w_{tm+1}$ is determined uniquely by 
\begin{align} \label{Eq: vv=vw}
v_{tm+1} \wedge v_{tm+2} = v_{tm+2} \wedge w_{tm+1},  \quad w_{tm+1} \in {\rm span}\{ v_{tm+3}, \ldots, v_{(t+1)m+1} \}.
\end{align}
Here we understand $ v_{N+1} = v_1$ (see \cite[Definition 5.2]{Fra20}). 
We identify the $m$-fold exterior product $ \wedge^m(V) $ with $\C$ satisfying $\mathrm{e}_1 \wedge \mathrm{e}_2 \wedge \cdots \wedge \mathrm{e}_m = 1$, where $\mathrm{e_k}$ is the standard $k$-th unit vector.
Explicitly, we have 
\begin{align} \label{Eq: w}
w_{tm+1} = \frac{\det( v_{tm+1}, v_{tm+3}, \ldots, v_{(t+1)m+1} )}{\det( v_{tm+2}, v_{tm+3}, \ldots, v_{(t+1)m+1} )} v_{tm+2} - v_{tm+1} 
\end{align}
(see \cite[Remark 5.6]{Fra20}).
Then the action $\tsig_1$ is defined by pullback by the map $\kappa$ in the coordinate ring $\C[ \toGr(m,N)]$, i.e., 
$$
(\tsig_1(x))( \textbf{v} ) = x( \kappa(\textbf{v})) \quad \text{ for $x \in \C[ \toGr(m,N)]$ and $ \textbf{v} \in\toGr(m,N)$.}
$$

We now introduce some notations.
For $a\in [1, m]$ and $b \in \Z$ with $b-a+1 \in 2\Z$, define 
\begin{equation} \label{Eq: bfi}
\begin{aligned} 
\bfi_a &:= (1,2, \ldots, a, a+2, a+3, \ldots, m+1) \in \Z^m , \\
\bfi_{a, b} &:= \bfi_a + \frac{b-a+1}{2}  \in \Z^m.
\end{aligned}
\end{equation}
For simplicity, we sometimes write 
$$ 
\vPab{a,b} := \vP_{ \bfi_{a,b} }.
$$

Define 
\begin{equation} \label{Eq: SABCD}
\begin{aligned} 
S &:= \{ (a, b) \in [1,m] \times \Z \mid  a-1 \le b \le a+2m-3, \ b-a \equiv 1\ \textrm{mod}\ 2    \} \\ 
& \quad \  \cup \{ (m, 3m-1) \}, \\
A &:= \{ (a,a+1) \in S \mid a \in [1,m-2]\}, \\
B & := \{ (m-1, m) \}, \\ 
C & := \{ (a,b) \in S \mid a+b = 2m-1,\ b > m\}, \\
D & := \{ (a,b) \in S \mid a=m \}. 
\end{aligned}
\end{equation}

Recall the set $\crhI_\qQ$ given in \eqref{Eq: S S+ SQ}.

\begin{example} \label{Ex: Pab}
Let $m=4$. Then $\bfi_{a,b}$ ($(a,b)\in S$) are given as follows:
\begin{align*}
\bfi_{1,0} &= (1,3,4,5), \quad \bfi_{1,2} = (2,4,5,6),\quad \bfi_{1,4} = (3,5,6,7),\quad \bfi_{1,6} = (4,6,7,8), \\
\bfi_{2,1} &= (1,2,4,5), \quad \bfi_{2,3} = (2,3,5,6),\quad \bfi_{2,5} = (3,4,6,7),\quad \bfi_{2,7} = (4,5,7,8), \\
\bfi_{3,2} &= (1,2,3,5), \quad \bfi_{3,4} = (2,3,4,6),\quad \bfi_{3,6} = (3,4,5,7),\quad \bfi_{3,8} = (4,5,6,8), \\
\bfi_{4,3} &= (1,2,3,4), \quad \bfi_{4,5} = (2,3,4,5),\quad \bfi_{4,7} = (3,4,5,6),\quad \bfi_{4,9} = (4,5,6,7), \quad \bfi_{4,11} = (5,6,7,8).
\end{align*}
Pictorially, the elements $ \vPab{a,b} $ ($(a,b) \in S \setminus D$) are written as follows: 
\begin{equation} \label{Eq: ABC}
\begin{aligned} 
\scalebox{0.8}{\xymatrix@C=1ex@R=0.0ex{
\vPab{1,0}  && \ul{\vPab{1,2}} && \vPab{1,4} && \dul{\vPab{1,6}} \\
& \vPab{2,1} && \ul{\vPab{2,3}} && \dul{\vPab{2,5}} && \vPab{2,7} \\
&& \vPab{3,2} &&  \tul{\vPab{3,4}} && \vPab{3,6} && \vPab{3,8} \\
}}
\end{aligned} 
\end{equation}
and $ \vPab{a,b} $ ($(a,b) \in D$) are written as follows:
\begin{align} \label{Eq: D}
\scalebox{0.8}{\xymatrix@C=1ex@R=0.0ex{
\vPab{4,3} && \vPab{4,5} && \vPab{4,7} && \vPab{4,9} && \vPab{4,11}. 
}}
\end{align}
Then $A$ is the set of the underlined elements, $B$ is the set of the triple underlined element, $C$ is the set of double underlined elements in \eqref{Eq: ABC}.    
Note that $\vP_\bfi$'s in \eqref{Eq: D} are frozen variables (see Figure \ref{fig:an initial quiver for iGr4plus}). 
It is easy to see that, when $n=3$, 
$$
S \setminus D = \crhI_\qQ \cup \cdual (\crhI_\qQ),
$$ 
where $\crhI_\qQ$ is given in \eqref{Eq: S S+ SQ} in Section \ref{Sec: ed An} 
(see Example \ref{Ex: pic for crhI}).
\end{example}

Let us consider the commuting diagram 
\begin{equation} \label{Eq: diagram}
\begin{aligned} 
\xymatrix{
\C[\tdGr(m, N)] \ar[rr]^{ \qquad \  \sim}_{\qquad \quad \psi_{m, N}} \ar[drr]_\pi && \iGr_m^{N} \ar@{->>}[d]^p \\ 
  && \biGr_m^{N},
}
\end{aligned}
\end{equation}
where $\psi_{m, N}$ is the isomorphism given in Lemma \ref{Lem: phimN} and the vertical map $p$ is the canonical projection. We denote by $\pi$ the composition $p \circ \psi_{m, N}$. In the following lemma, we identify $\iGr_m^{N}$ with $ \C[\tdGr(m,N)]$ under the isomorphism $\psi_{m, N}$ and understand $ \C[\tdGr(m,N)]$ as a subalgebra of $\C[\toGr(m,N)]$.

\begin{lemma} \label{Lem: tsig}
Suppose that $m$ divides $N$ and $N/m \ge 3$.
Let $w := s_1 s_{m+1} \in \sg_N$.
Then we have the following.
\bnum
\item 
For any $(a,b) \in S$, the image $\tsig_1( \vP_{\bfi_{a,b}} ) $ is contained in $\C[\tdGr(m,N)]$ and 
$$
\pi( \tsig_1 ( \vP_{\bfi_{a,b}} )) = 
\begin{cases}
 \vP_{\bfi_{a+1,a}}   & \text{ if } (a,b) \in A, \\
\vP_{\bfi_{1,2m}} & \text{ if } (a,b) \in B, \\
\vP_{\bfi_{1,2m}} \cdot \vP_{\bfi_{a,b}} - \vP_{   \bfi_{a+1,b+1} }   & \text{ if } (a,b) \in C, \\
1 & \text{ if } (a,b) \in D, \\
 \vP_{ w(\bfi_{a,b})} & \text{ otherwise,}\\
\end{cases}
$$
where the sets $S$, $A$, $B$, $C$ and $D$ are given in \eqref{Eq: SABCD} and $\pi$ is the surjection given in \eqref{Eq: diagram}. 
\item  \label{Lem: tsig (ii)}
Suppose that $N/m \ge 4 $. For any $(a,b) \in S$ and $j \in [1,m-1]$, we have
$$
\pi ( \tsig_j(\vP_{\bfi_{a,b}}) ) =  \pi (  \sh_{ j-m-1 } \circ \  \tsig_1 \circ \sh_{m+1-j} (\vP_{\bfi_{a,b}}) ).
$$  
\ee
\end{lemma}
\begin{proof}

We set
\begin{align*}
 \cvC{m}_N &:= \{ \bfi = (i_1, i_2, \ldots, i_m) \in \cvC{m} \mid \bfi \subset [1,N] \},\\
\cvF{m}_N &:= \{ \bfi = (i_1, i_2, \ldots, i_m) \in \cvF{m} \mid \bfi \subset [1,N] \}.
\end{align*}
For simplicity, we write 
\begin{itemize}
\item $j \in \bfi$ if $j$ appears in $\bfi$ as a component of $\bfi$, 
\item $f \eF g$ if $\pi(f) = \pi(g)$ for $ f,g \in \C[\tdGr(m, N)]$.
\end{itemize}

\smallskip

Let $\bfv = (v_1, v_2, v_3 \ldots, v_{m+1}, v_{m+2}, v_{m+3},  \ldots, v_N) \in \mathrm{SL}(V) \backslash (V^N)^\circ$. 
Note that $\Delta_{\bfi} (\bfv) = \det( v_{i_1},v_{i_2}, \ldots, v_{i_m} )$ for $\bfi = (i_1, i_2, \ldots, i_m) \in \cvC{m}_N$.
We set 
$$
\bfv' := \kappa(\bfv) = ( v_2, w_1, v_3, \ldots, v_{m+2}, w_{m+1}, v_{m+3}, \ldots, v_N ),
$$
where 
\begin{align*} 
w_{1} = \frac{ \Delta_{\bfi_{1,0}} (\bfv) }{\Delta_{\bfi_{m,m+1}}(\bfv) } v_{2} - v_{1}, \qquad  
w_{m+1} = \frac{ \Delta_{\bfi_{1,2m}} (\bfv) }{\Delta_{\bfi_{m,3m+1}}(\bfv) } v_{m+2} - v_{m+1}.
\end{align*}
(see \eqref{Eq: w}). Note that $\Delta_{\bfi_{m,m+1}}  \eF  \Delta_{\bfi_{m,3m+1}}  \eF 1 $
since $\bfi_{m,m+1}, \bfi_{m,3m+1} \in \cvF{m}_N$.

We now shall prove (i) by case-by-case check.

\smallskip
\noindent
 \textsf{(Case: $(a,b) \in S \setminus ( A\cup B\cup C \cup D )$)}

In this case, $\bfi_{a,b}$ satisfies the following conditions:
\bna
\item if $2 \in \bfi_{a,b}$, then $1 \in \bfi_{a,b}$,
\item if $m+2 \in \bfi_{a,b}$, then $m+1 \in \bfi_{a,b}$.
\ee
By \eqref{Eq: vv=vw}, we have 
$ \tsig_1 (\Delta_{\bfi_{a,b}}) (\bfv)= \Delta_{\bfi_{a,b}} (\bfv') =   \Delta_{ w (\bfi_{a,b})} (\bfv)$.

\smallskip
\noindent
\textsf{(Case: $(a,b) \in A$)}

In this case, we have $b=a+1$. Then  $1 \notin \bfi_{a,a+1}$ and $2, m+1,m+2 \in \bfi_{a,a+1}$, i.e., 
$$
\bfi_{a,a+1} = (2, \ldots, m+1, m+2).
$$  By definition, we have 
\begin{align*}
\tsig_1 (\Delta_{\bfi_{a,a+1}}) (\bfv)  = \Delta_{\bfi_{a,a+1}} (\bfv')
\eF \Delta_{\bfi_{1,0}} (\bfv) \cdot \Delta_{\bfi_{a,a+1}} (\bfv) - \Delta_{s_1(\bfi_{a,a+1})} (\bfv) \eF \Delta_{\bfi_{a+1,a}} (\bfv),
\end{align*}
where the last equality follows by applying 
$\bfi = (3,4, \ldots, m+1) \in \cvC{m-1}$ and $\bfj = (1,2,\ldots, a+1, a+3, \ldots, m+1,m+2) \in \cvC{m+1}$ to the Pl\"{u}cker relation \eqref{Eq: Plucker}.

\smallskip
\noindent
\textsf{(Case: $(a,b) \in B$)}

In this case, $\bfi_{m-1,m} = (2,3,  \ldots, m, m+2)$. Thus we have 
\begin{align*}
\tsig_1 (\Delta_{\bfi_{m-1,m}}) (\bfv)  &= \Delta_{\bfi_{m-1,m}} (\bfv') \\
&\eF \Delta_{\bfi_{1,0}} (\bfv) \Delta_{\bfi_{1,2m}} (\bfv) \Delta_{\bfi_{m-1,m}} (\bfv) 
- \Delta_{\bfi_{1,0}} (\bfv) \Delta_{\bfi_{m,m+1}} (\bfv) \\
& \quad \ - \Delta_{\bfi_{1,2m}} (\bfv) \Delta_{s_1 (\bfi_{m-1,m})} (\bfv)
+ \Delta_{\bfi_{1,0}} (\bfv)
\\ 
&\eF \Delta_{\bfi_{1,2m}} (\bfv) \cdot \left( \Delta_{\bfi_{1,0}} (\bfv) \cdot  \Delta_{\bfi_{m-1,m}} (\bfv) 
- \Delta_{ s_1( \bfi_{m-1,m} )} (\bfv)   \right) \\
&\eF \Delta_{\bfi_{1,2m}} (\bfv),
\end{align*}
where the last equality follows by applying 
$\bfi = (3,4, \ldots, m+1) \in \cvC{m-1}$ and $\bfj = (1,2,\ldots, m,m+2) \in \cvC{m+1}$ to the Pl\"{u}cker relation \eqref{Eq: Plucker}.

\smallskip
\noindent
\textsf{(Case: $(a,b) \in C$)}

In this case, $\bfi_{a,b} = (b-m+2,  \ldots, m, m+2, \ldots)$. Since $ w (\bfi_{a,b}) = s_{m+1} (\bfi_{a,b})= \bfi_{a+1,b+1}$, we have 
\begin{align*}
\tsig_1  (\Delta_{\bfi_{a,b}}) (\bfv)  = \Delta_{\bfi_{a,b}} (\bfv')
\eF \Delta_{\bfi_{1,2m}} (\bfv) \cdot \Delta_{\bfi_{a,b}} (\bfv) - \Delta_{\bfi_{a+1,b+1}} (\bfv).
\end{align*}

\smallskip
\noindent
\textsf{(Case: $(a,b) \in D$)}

In this case, we have $\bfi_{a,b} \in \cvF{m}_N$. Thus it follows from \cite[(19)]{Fra20}.

(ii) 
Let $\rho^*$ be the pullback of the \emph{twisted cyclic shift} $\rho$ given in \cite[Section 5]{Fra20}.
Let $(a,b) \in S$. Since $N/m \ge 4$, we have 
\begin{align*}
\sh_1( \Delta_{\bfi_{a,b}} ) = \rho^*(\Delta_{\bfi_{a,b}}), \qquad 
(\rho^*)^m \circ \tsig_j (  \Delta_{\bfi_{a,b}} ) =   \tsig_j \circ (\rho^*)^m (\Delta_{\bfi_{a,b}})
\end{align*} 
It follow the assertion form the fact that $ (\rho^*)^{j-1} \circ \tsig_1 \circ (\rho^*)^{1-j} = \tsig_j $. 
\end{proof}

\begin{prop} \label{prop: sigma_i}  For each $i\in [1,m-1]$, there is a homomorphism $\sigma_i:\biGr_{m}^+ \rightarrow \biGr_{m}^+$ such that 
\bna
\item the homomorphisms $\sigma_i$ satisfy the braid group relation,
\item \label{prop: sigma_i (b)}  $ \sigma_i \circ \sh_{m} = \sh_{m} \circ \sigma_i$,
\item given any $x\in \iGr_{m}^+$, there is a positive integer $N$ such that
$$
\sigma_i \circ p (x) = p \circ  \tilde{\sigma}_i^{(m, N')} (x) \quad \text{ for any $N' \ge N$ with $m \mid N'$,} 
$$
where $p: \iGr_{m}^+ \twoheadrightarrow\biGr_m^+$ is the canonical projection.
\ee
\end{prop}
\begin{proof}
Let $x \in \biGr_{m}^+$ and let $r\in \Z_{>0}$ such that $x$ is contained in the image of 
the map $\iota: \biGr_{m}^r \to \biGr_{m}^+$ in the colimit system. Choose $N\in \Z_{>0}$ such that $m$ divides $N$ and $N>3m+r$.

Recall Fraser's braid group action $\tilde{\sigma}_i$ on the localization of $\C[ \toGr(m,N)]$.
Define $\sigma_1 (x) = \pi (\tilde{\sigma}_1(x_\circ))$, where $x_\circ \in \C[\tdGr(m, N)]$ such that  $\pi(x_\circ)=x$. Since $ \tilde{\sigma}_1$ takes frozens to frozens and  
$\tilde{\sigma}_1(x_\circ)$ is stable under taking any $N > 3m+r$ with $m | N$, $\sigma_1$ is well-defined.
As $\tilde{\sigma}_1$ is an algebra homomorphism, $\sigma_1$ is an algebra homomorphism. For $ k \in [2, m-1]$, define
\begin{align} \label{Eq: sigma_j}
\sigma_k := \sh_{k-m-1} \circ \sigma_1 \circ \sh_{m+1-k}. 
\end{align}

(a) Since $\tilde{\sigma}_i$ satisfy the braid group actions, Lemma \ref{Lem: tsig} \eqref{Lem: tsig (ii)} tells us that $\sigma_i$ satisfy the braid relations. 

(b) Since $\tilde{\sigma}_i$ commutes with $\sh_m$, we have
 $ \sigma_i \circ \sh_{m} = \sh_{m} \circ \sigma_i$. 

(c) This follows from the definition of $\sigma_i$. 
\end{proof}

\begin{remark} \label{Rmk: inf Gr}
One can define the braid group action $\sigma_i$ in terms of the configuration of vectors in $\C^m$ following \cite[Section 5]{Fra20}.
We modify Fraser's braid group action on $\tGr(m,N)$ (see Definition 5.2 in \cite{Fra20}) to give a braid group action $\tGr(m, \infty)$. For $i \in [m-1]$, define a map $\kappa_i: \tGr(m, \infty) \to \tGr(m, \infty)$ sending a $m\times \infty$ matrix to a $m\times \infty$ matrix as follows. Denote by $(v_1,v_2,\ldots)$ a matrix in $\tGr(m, \infty)$, where $v_i$'s are columns and the tuple $(v_1,v_2,\ldots)$ is consecutively generic. The matrix $(v_1,v_2,\ldots)$ is divided into windows $[v_1, \ldots, v_m]$, $[v_{m+1},\ldots, v_{2m}]$, $\ldots$. Under the map $\kappa_i$, the first window is changed to $[v_1, \ldots, v_{i-1}, v_{i+1}, w_1, v_{i+2}, \ldots, v_m]$, where $w_1$ is determined uniquely by 
\begin{align}
v_i \wedge v_{i+1} = v_{i+1} \wedge w_1,  \quad w_1 \in {\rm span}\{ v_{i+2}, \ldots, v_{i+m} \}.
\end{align}
Explicitly,
\begin{align*}
w_1 = \frac{\det( v_i, v_{i+2}, \ldots, v_{i+m} )}{\det( v_{i+1}, v_{i+2}, \ldots, v_{i+m} )} v_{i+1} - v_i. 
\end{align*}
The other windows are changed in the same way as the the first window. 
Then $\sigma_i$ is the pullback of $\kappa_i$.
\end{remark}

For any $i\in [1,m-1]$ and $x \in \biGr_{m}$, we define 
\begin{align} \label{Eq: def of sig}
\sigma_i(x) := \sh_{- mt} \circ  \sigma_i \circ \sh_{mt }(x)
\end{align}
for a sufficiently large positive integer $ t \gg 0$. Proposition \ref{prop: sigma_i} tells us that it is well-defined, thus we extend $\sigma_i$ to the whole algebra $\biGr_{m}$, i.e.,
\begin{align} \label{Eq: sigmai}
\sigma_i : \biGr_{m} \longrightarrow \biGr_{m}.
\end{align}
By Proposition \ref{prop: sigma_i} \ref{prop: sigma_i (b)}  and \eqref{Eq: sigma_j}, for any $k\in [1,m-1]$, we have 
\begin{equation} \label{Eq: sigk}
\begin{aligned}
\sigma_k \circ \sh_{m} &= 	 \sh_{m} \circ  \sigma_k  \\
\sigma_k &= 	\sh_{k-1} \circ \sigma_1 \circ \sh_{1-k} 
\end{aligned}	
\end{equation}

Let $\sg_{+\infty} := \cup_{n > 0} \sg_n$, where we understand $\sg_n$ as a subgroup of $\sg_{n+1}$ generated by permutations $\tau$ such $\tau(n+1) = n+1$. Lemma \ref{Lem: tsig} gives the following. 

\begin{lemma} \label{Lem: sigma1 on SD}
Let $w := s_1 s_{m+1} \in \sg_{+\infty}$.
For any $(a,b) \in S \setminus D$, we have
$$
  \sigma_1( \vP_{\bfi_{a,b}} ) = 
\begin{cases}
	 \vP_{\bfi_{a+1,a}}     & \text{ if } (a,b) \in A, \\
	\vP_{\bfi_{1,2m}} & \text{ if } (a,b) \in B, \\
	\vP_{\bfi_{1,2m}} \cdot \vP_{\bfi_{a,b}} - \vP_{   \bfi_{a+1,b+1} }   & \text{ if } (a,b) \in C, \\
	\vP_{ w(\bfi_{a,b})} & \text{ otherwise,}
\end{cases}
$$
where the sets $S$, $A$, $B$, $C$ and $D$ are given in \eqref{Eq: SABCD}. 
\end{lemma}

\vskip 2em 

\section{Comparison of braid group actions} \label{sec:compare braid group actions}

In this section, we shall compare the braid group actions $\bT_i$, $\bR_i$ and $\sigma_i$. We keep the notations appearing in Section \ref{Sec: ed An}.

Lemma \ref{Lem: Ti respects B} \ref{Lem: Ti respects B (iii)} says that the action $\bT_i$ induces a permutation on the categorical crystal $\sB(\g)$. Thus it makes sense to compare $\bT_i$ and $\bR_i$ as permutations of $\sB(\g)$. It is shown in \cite{KKOP24A} that $\bT_i$ coincides with $\bR_i$ using the global basis of the bosonic extension $\qAg$.

\begin{theorem}[{\cite[(3.15)]{KKOP24A}}] \label{Thm: Ti=Ri}
For each $i\in I_0$, we have $\bT_i = \bR_i$ on the categorical crystal $\sB(\g)$.
\end{theorem}

Recall the infinite Grassmannian cluster algebra $\biGr_{m}$ and $\biGr_{m}^+$ given in Section \ref{Sec: igca}, and the sequence $\bfi_{a,b}$ given in \eqref{Eq: bfi}. We set 
$m:=n+1$.

\begin{lemma} \label{lemma: Phi+} \
\bni
\item 
There is an algebra isomorphism 
$$
\KG_n^+ : K(\catCP)  \buildrel \sim \over \longrightarrow \biGr_{n+1}^+ 
$$
sending $ [V(Y_{i, a})]$ to $\vP_{ \bfi_{i,a} }$ for any $ (i, a) \in \crhI_{\Z_{\ge0}} $.

\item For any $k\in \Z_{>0}$,  we have 
$$
\KG_n^+ \circ \shift_k  =  \sh_k \circ \KG_n^+ 
$$
In particular, we have $\KG_n^+ \circ [\dual^2] = \sh_{n+1} \circ \KG_n^+$, 
where $\dual$ is the right dual functor in $\catCO$. 
\ee 
\end{lemma}
\begin{proof}
(i) It follows from \cite[Lemma 3.9]{CDFL20}.

(ii) It follows from (i) directly.
\end{proof}

\begin{example} \label{Ex: VY P}
We keep the notations in Example \ref{Ex: pic for crhI}.

The correspondence between $[V(Y_{i, a})]$ and $\vPab{  i,  a }$ under the isomorphism $\KG_n^+$ is given pictorially as follows (see also Example \ref{Ex: Pab} and Example \ref{Ex: pic for [a,b]}).
$$ 
\scalebox{0.8}{\xymatrix@C=1ex@R=  0.0ex{ 
i \diagdown\, a     \ar@{-}[dddd]<3ex> \ar@{-}[rrrrrrrrrrrrrrrr]<-2.0ex>    & -1 & 0 & 1 & 2 & 3 & 4 & 5 & 6 & 7 & 8 &9&10 & 11& 12 & & &   \\
1       & &  \vPab{1,0} & & \vPab{1,2}  &  & \vPab{1,4}  & &  \vPab{1,6} & & \vPab{1,8} && \vPab{1,10} && \vPab{1,12}   \\
2     &  & & \vPab{2,1} & & \vPab{2,3} &  &  \vPab{2,5}  & &  \vPab{2,7}  && \vPab{2,9} && \vPab{2,11}  \\
3      & &  & & \vPab{3,2} & &  \vPab{3,4}  & &  \vPab{3,6} && \vPab{3,8} && \vPab{3,10} && \vPab{3,12} \\
& 
}}
$$	
\end{example}

For any element $x \in K(\catCO)$, we define 
$$
\KG_n(x) := \sh_{- k} \circ  \KG_n^+ \circ \shift_k (x)
$$
for a sufficiently large positive integer $ k \gg 0$. By Lemma \ref{lemma: Phi+}, it is well-defined. We thus extend $\KG_n^+$ to an algebra isomorphism of the whole algebras, which denoted by $\KG_n$, i.e., 
\begin{align} \label{Eq: Phi}
\KG_n : K(\catCO)  \buildrel \sim \over \longrightarrow \biGr_{n+1}. 
\end{align}
Lemma \ref{lemma: Phi+} says that  
\begin{align} \label{Eq: Phin sh}
	\KG_n \circ \shift_k  =  \sh_k \circ \KG_n\qquad \text{ for $k\in \Z.$} 
\end{align}

Since the braid group $\bT_i$ respects to $\sB(\g)$ (see Lemma \ref{Lem: Ti respects B}) and the $(q,t)$-characters go to $q$-characters when specializing $t=1$, 
$\bT_i$ can be regarded as an automorphism of $ K(\catCO)$.

\begin{lemma} \label{Lem: T_1 [a,b]}\ 
\bni
\item For any segment  $[a,b]$ and $k\in \Z$, we have 
$$ 
\bT_1 \left( [ V([a,b]_k) ] \right) = 
\begin{cases}
    [ V([1]_{k+1})  & \text{ if $a=b = 1$}, \\
    [V([1]_{k+1})] \cdot [V([1,b]_k)] - [V([2,b]_k)]  & \text{ if $a = 1$ and $b>1$}, \\
    [ V([1,b]_k) ] & \text{ if $a = 2$},  \\
    [ V([a,b]_k) ] & \text{ otherwise}.
\end{cases}
$$
\item Let $\crhI_\qQ$ be given in \eqref{Eq: S S+ SQ}. For any $(i,p) \in \crhI_\qQ$ and $k\in \Z$, we have 
$$ 
\bT_1 \left( [ V(Y_{\cdual^k(i,p)}) ] \right) = 
\begin{cases}
[ V( Y_{\cdual^k(n,n+1)})]  & \text{ if $i=1, p=0$}, \\
[ V( Y_{\cdual^k(n,n+1)})] \cdot [ V( Y_{\cdual^k(i,i-1)})] & \smash{\raisebox{-1.6ex}{\text{ if $p = i-1$ and $i>1$}}} \\
 \quad  - [ V( Y_{\cdual^k(i-1,i)})] \\
[ V( Y_{\cdual^k(i+1,i)})] & \text{ if $p = i+1$},  \\
[ V( Y_{\cdual^k(i,p)})] & \text{ otherwise}.
\end{cases}
$$
\ee
\end{lemma}	
\begin{proof}
(i) It follows from Example \ref{Ex: cB(inf)}, \eqref{Eq: R_iQ}, \eqref{Eq: Vm} and Theorem \ref{Thm: Ti=Ri} that 
$$ 
\bT_1 \left( [ V([a,b]_k) ] \right) = 
\begin{cases}
	[ V([1]_{k+1})  & \text{ if $a=b = 1$}, \\
[ V([1]_{k+1}) \htens V([1,b]_k) ] & \text{ if $a = 1$ and $b>1$}, \\
[ V([1,b]_k) ] & \text{ if $a = 2$},  \\
[ V([a,b]_k) ] & \text{ otherwise}.
\end{cases}
$$
Note that, for $b>1$, $ V([1,b]_k) \simeq  V([1]_k) \htens  V([2,b]_k)$ by Dorey's rule (see \cite[Lemma B.1]{AK97} for example).
By \cite[Proposition 3.16]{KKOP20} and \cite[Lemma B.2]{AK97}, we have
$$ 
\de ( V([1]_{k+1}  ),  V([1,b]_k ) ) = 1
$$
(see \cite[Section 3]{KKOP20} for the definition of $\de$).
Since $ V([1,b]_k) \htens V([1]_{k+1}) \simeq  (V([1]_k) \htens  V([2,b]_k)) \htens V([1]_{k+1}) \simeq  V([2,b]_k) $ (see \cite[Lemma 2.5]{KP22} for example), by \cite[Proposition 4.7]{KKOP20}, we have the following exact sequence
$$
\xymatrix{
0 \ar[r] & V([2,b]_k) \ar[r] & V([1]_{k+1}) \tens V([1,b]_k) \ar[r] &  V([1]_{k+1}) \htens V([1,b]_k) \ar[r] & 0. 
}
$$
Thus we have $ [V([1]_{k+1}) \htens V([1,b]_k)] = 	[V([1]_{k+1})] \cdot [V([1,b]_k)] - [V([2,b]_k)]$.

(ii) By \eqref{Eq: [a,b]}, we know that $[V(Y_{ \cdual^k(i,p)})]$ corresponds to $V([(p-i+3)/2, (p+i+1)/2]_k)$
via the isomorphism $\phi_\qQ$. Thus the statement follows from (i) directly.
\end{proof}

\begin{theorem} \label{Thm: T=sig}
For any $i\in I_0$, we have $ \KG_n \circ \bT_i = \sigma_i \circ \KG_n $, i.e., the following diagram commutes:
$$
\xymatrix{
K(\catCO) \ar[d]_{\bT_i}  \ar[rr]^\sim_{\KG_n} && \biGr_{n+1} \ar[d]^{\sigma_i} \\
K(\catCO) \ar[rr]^\sim_{\KG_n} && \biGr_{n+1} .
}
$$
\end{theorem}
\begin{proof}
For notational simplicity, we write 
$$
V(\cdual^k(i,p)) :=  [ V(Y_{\cdual^k(i,p)}) ].
$$
By Lemma \ref{Lem: Ti respects B}, \eqref{Eq: sigk} and \eqref{Eq: Phin sh}, it suffices to show the statement holds for $i=1$. Recall $\crhI_\qQ$ (see \eqref{Eq: S S+ SQ}) and $S$, $A$, $B$, $C$ and $D$ (see \eqref{Eq: SABCD}). 
It is easy to see that 
$$
\bbS := S \setminus D = \crhI_\qQ \cup \cdual(\crhI_\qQ).
$$
By \eqref{Eq: Sh T= T Sh} and \eqref{Eq: sigk}, it is enough to show that 
\begin{align} \label{Eq: Phi T sig Phi}
\KG_n \circ \bT_1 ( V(\cdual^k(i,p))  ) = \sigma_1 \circ \KG_n (V(\cdual^k(i,p)) ) \qquad\text{ for any $(i,p) \in \crhI_\qQ$ and $k=0,1$.}
\end{align}

Recall $m=n+1$. Let 
$$
 w := s_1 s_{m+1} \in \sg_{+\infty}.
$$ 
We shall prove by case-by-case check. 
Lemma \ref{Lem: sigma1 on SD} and Lemma \ref{Lem: T_1 [a,b]} will be used frequently in the following checks, but for more readability, we don't mention these lemmas in the following checks.

\noindent
\textbf{(Case $k=0$)} Let $(i,p) \in \crhI_\qQ$.
\bni
\item If $ p > i+1  $, then we have
\bna
\item $(i,p) \in \bbS \setminus (A\cup B\cup C) $,
\item $1,2 \notin \bfi_{i,p}$ and  $m+1,m+2 \in \bfi_{i,p}$, i.e., $w(\bfi_{i,p}) = \bfi_{i,p} $.
\ee
We have that 
$$
\KG_n \circ \bT_1 ( V(i,p) ) = \KG_n  ( V(i,p) ) 
= \vP_{ \bfi_{i,p} } = \vP_{ w(\bfi_{i,p}) } = \sigma_1 \circ \KG_n ( V(i,p) ).
$$

\item Assume $ p = i+1 $. Since $(i,p) \in A $, we obtain  
$$
\KG_n \circ \bT_1 ( V(i,i+1) ) = \KG_n  ( V(i+1,i) ) 
= \vP_{ \bfi_{i+1,i} } =  \sigma_1 \circ \KG_n ( V(i,i+1) ).
$$
\item Assume $ p = i-1$. It is easy to see that $(i,p) \in \bbS \setminus (A\cup B\cup C) $. 

If $i=1$, then $w(\bfi_{1,0}) = \bfi_{m-1,m} = \bfi_{n,n+1}$. Thus we have 
$$
\KG_n \circ \bT_1 ( V(1,0) ) = \KG_n  ( V(n,n+1) ) 
= \vP_{ \bfi_{n,n+1} } =  \sigma_1 \circ \KG_n ( V(1,0) ).
$$

If $i>1$, then 
\begin{align*}
 \bfi_{i,i-1} &= (1,2, \ldots, i, i+2, \ldots,  m+1), \\
\bfi_{k-1,k} &= (2, \ldots, k, k+2, \ldots,  m+2) \qquad (k=i,m), \\
 w(\bfi_{i,i-1}) &= (1,2, \ldots, i, i+2 + \delta_{i, m-1}, \ldots, m+2).
\end{align*}
Then we have 
\begin{align*}
\KG_n \circ \bT_1 ( V(i,i-1) ) = \vP_{ \bfi_{m-1,m} } \vP_{ \bfi_{i,i-1} } - \vP_{ \bfi_{i-1,i} } =  \vP_{ w(\bfi_{i,i-1}) } = \sigma_1 \circ \KG_n ( V(i,i-1) ),
\end{align*}
where the second identity follows by 
applying 
$\bfi = (2,3, \ldots i,i+2 + \delta_{i, m-1},\ldots,  m+2) \in \cvC{m-1}$ and $\bfj = (1,2,\ldots, m+1) \in \cvC{m+1}$ to the Pl\"{u}cker relation \eqref{Eq: Plucker}.
\ee

\noindent
\textbf{(Case $k=1$)} Let $(i,p) \in \crhI_\qQ$. Note that $\cdual(i,p) = (n+1-i, p+n+1)$.
\bni
\item If $ p > i+1  $, then we have
\bna
\item $\cdual(i,p) \in \bbS \setminus (A\cup B\cup C) $,
\item $1,2 \notin \bfi_{\cdual(i,p)}$ and  $m+1,m+2 \in \bfi_{\cdual(i,p)}$, i.e., $w(\bfi_{\cdual(i,p)}) = \bfi_{\cdual(i,p)} $.
\ee
We have that 
$$
\KG_n \circ \bT_1 ( V( \cdual(i,p)) ) = \KG_n  ( V(\cdual(i,p)) ) 
= \vP_{ \bfi_{\cdual(i,p)} } = \vP_{ w(\bfi_{\cdual(i,p)}) } = \sigma_1 \circ \KG_n ( V(\cdual(i,p)) ).
$$

\item Assume $ p = i+1 $. Note that $\cdual(i,i+1) \in \bbS \setminus (A\cup B\cup C) $ and 
\begin{align*}
\bfi_{\cdual(i+1, i)} &= \bfi_{n-i,i+n+1} = 
(i+2, i+3, \ldots, m, m+2, \ldots, m+i+2), \\ 
w (\bfi_{\cdual(i+1, i)}) &=   
(i+2, i+3, \ldots, m, m+1, m+3, \ldots, m+i+2)
= \bfi_{n-i+1,i+n+2}.
\end{align*}

Thus we have 
\begin{align*}
\KG_n \circ \bT_1 ( V( \cdual( i,i+1)) ) &= \KG_n  ( V(\cdual(i+1,i)) ) 
= \vP_{ \bfi_{\cdual(i+1,i)} } =  
\vP_{ w(\bfi_{n-i+1,i+n+2}) } \\
&= \sigma_1 \circ \KG_n ( V(n-i+1,i+n+2) )
= \sigma_1 \circ \KG_n ( V( \cdual ( i,i+1) ).
\end{align*}

\item Assume $ p = i-1$. It is easy to see that $\cdual(i,p) \in B\cup C $. 

If $i=1$, then $\cdual(1,0) \in B$. Thus we have 
\begin{align*}
\KG_n \circ \bT_1 ( V( \cdual( 1,0)) ) &= \KG_n  ( V(\cdual( n,n+1)) ) 
= \vP_{ \bfi_{ \cdual( n,n+1 )} } 
= \vP_{ \bfi_{  1,2(n+1) } } \\
& =  \sigma_1 \circ \KG_n ( V(n,n+1) ) =  \sigma_1 \circ \KG_n ( V( \cdual(1,0)) ).
\end{align*}

If $i>1$, then $\cdual(i,i-1) \in C$. Thus we have 
\begin{align*}
\KG_n \circ \bT_1 ( V( \cdual( i,i-1)) ) &= 
 \vP_{ \bfi_{ \cdual( n,n+1 )} } \vP_{ \bfi_{ \cdual( i,i-1 )} }  - \vP_{ \bfi_{ \cdual( i-1,i )} } \\
& = \vP_{ \bfi_{   1,2m } } \vP_{ \bfi_{   m-i,m+i-1 } }  - \vP_{ \bfi_{  m-i+1,m+i } } \\
& =  \sigma_1 \circ \KG_n ( V(m-i,m+i-1) ) \\
& =  \sigma_1 \circ \KG_n ( V( \cdual(i,i-1)) ).
\end{align*}
\ee
Therefore, we have the assertion.
\end{proof}

We define  
$$
\gB_{n+1} := \KG_n(\sB(\g)) \subset \biGr_{n+1}.
$$
By construction, $\gB_{n+1}$ is a basis of $\biGr_{n+1}$ with non-negative structure constants, and all cluster monomials are contained in $\gB_{n+1}$.
We then obtain the coincidence between $\bR_i$ and $\sigma_i$.

\begin{corollary} \label{Cor: R=sig}
For any $i\in I_0$, we have $ \KG_n \circ \bR_i = \sigma_i \circ \KG_n $ on the categorical crystal $\sB(\g)$, i.e., the following diagram commutes:
$$
\xymatrix{
\sB(\g) \ar[d]_{\bR_i}  \ar[rr]^{\KG_n} && \gB_{n+1} \ar[d]^{\sigma_i} \\
\sB(\g) \ar[rr]^{\KG_n} && \gB_{n+1}.
}
$$
\end{corollary}

\vskip 2em 

\section{Examples for rank \texorpdfstring{$2$}{2} cases} \label{sec:examples for rank 2 cases}

In this section, we illustrate the braid group actions by examples. 

 Let $U_q'(\g)$ be the quantum affine algebra of affine type $A_2^{(1)}$ and let $\qQ$ be the $Q$-datum given in Section \ref{Sec: ed An}. Note that 
\begin{align*}
\crhI_\Z =  \{ (1,k),\ (2, k+1) \mid k\in 2\Z   \} \quad \text{ and } \quad \crhI_\qQ =  \{ (1,2), (2,1), (1,0) \}.
\end{align*}
The isomorphism 
$$\KG_2 : K(\catCO)  \buildrel \sim \over \longrightarrow \biGr_{3}
$$
is given in \eqref{Eq: Phi}. In this subsection, we identify $K(\catCO)$ with $\biGr_{3}$ via the isomorphism $\KG_2$ if no confusion arises.

We recall the notations given in Remark \ref{Rmk: inf Gr}. Let $\bfv = (v_1,v_2, \ldots) \in \tGr(3, \infty)$ and define $\Delta_{i_1,i_2,i_3}(\bfv) = \det(v_{i_1},v_{i_2},v_{i_3})$.
We denote by $\pi_\circ : \C[\tGr(3, \infty)] \rightarrow \biGr_{3}^+$ the limit of the map $\pi$ given in \eqref{Eq: diagram} as $N \to \infty$. Then $\Delta_{i_1,i_2,i_3}$ goes to $\vP_{i_1,i_2,i_3}$ under the map $\pi_\circ$. For simplicity, we write 
$$
d(v_{i_1}, v_{i_2}, v_{i_3}) := \pi_\circ(  \Delta_{i_1,i_2,i_3}(\bfv) ) = \pi_\circ(  \det(v_{i_1},v_{i_2},v_{i_3}) )
$$
and identify $d(v_{i_1}, v_{i_2}, v_{i_3})$ with $\vP_{i_1,i_2,i_3}$.
For any $k\in \Z$, we set
\begin{align*}
\gt_1(v_k) &:= 
\begin{cases}
v_{k} & \text{ if } k \equiv 0 \mod 3, \\
v_{k+1} & \text{ if } k \equiv 1 \mod 3, \\
d(v_{k-1}, v_{k+1}, v_{k+2}) v_{k} - v_{k-1} & \text{ if } k \equiv 2 \mod 3,
\end{cases} \\
\gt_2(v_k) &:= 
\begin{cases}
d(v_{k-1}, v_{k+1}, v_{k+2}) v_{k} - v_{k-1} & \text{ if } k \equiv 0 \mod 3, \\
v_{k} & \text{ if } k \equiv 1 \mod 3, \\
v_{k+1} & \text{ if } k \equiv 2 \mod 3.
\end{cases}
\end{align*}
By the definition of $\sigma_i$ (see the beginning of Section \ref{subsec: bg action}, Proposition \ref{prop: sigma_i} and \eqref{Eq: def of sig}), we have the following lemma.
\begin{lemma} \label{Lem: P A2}
For any $(i_1, i_2, i_3) \in \cvC{3}$, we have 
\begin{align*}
\sigma_1(\vP_{i_1,i_2,i_3}) = d(\gt_1(v_{i_1}), \gt_1(v_{i_2}), \gt_1(v_{i_3})), \\
\sigma_2(\vP_{i_1,i_2,i_3}) = d(\gt_2(v_{i_1}), \gt_2(v_{i_2}), \gt_2(v_{i_3})).
\end{align*}
\end{lemma}

On the other hand, let us recall the extended crystal (see Section \ref{Sec: extended crystal}). Using the crystal description of the multisegment realization $\MS_n$ given in \cite[Section 7.1]{KP22}, for any $\bfm = a[2] + b[1,2] + c[1]$, we have 
\begin{align*}
 \tf_1(\bfm) &= 
 \begin{cases}
\bfm + [1] & \text{ if } a \le c,\\
\bfm -[2] + [1,2] & \text{ if } a > c,
\end{cases}
\qquad 
 \te_1(\bfm) = 
 \begin{cases}
0 & \text{ if $b=0$ and $a \ge c$},\\
\bfm -[1,2]+[2] & \text{ if $b>0$ and $a \ge c$},\\
\bfm -[1] & \text{ if $a < c$},
\end{cases}
\\
 \tf_2(\bfm) &= \bfm + [2],
\qquad \qquad \qquad \qquad \qquad 
 \te_2(\bfm) = 
 \begin{cases}
0 & \text{ if $a=0$},\\
\bfm - [2] & \text{ if $a > 0$},
\end{cases}
\end{align*}
and 
\begin{align*}
 \tfs_1(\bfm) &= \bfm + [1],
\qquad \qquad \qquad \qquad \qquad 
 \tes_1(\bfm) = 
 \begin{cases}
0 & \text{ if $c=0$},\\
\bfm - [1] & \text{ if $c > 0$},
\end{cases}
\\
 \tfs_2(\bfm) &= 
 \begin{cases}
\bfm + [2] & \text{ if } a \ge c,\\
\bfm -[1] + [1,2] & \text{ if } a < c,
\end{cases}
\qquad 
 \tes_2(\bfm) = 
 \begin{cases}
0 & \text{ if $b=0$ and $a \le c$},\\
\bfm - [1,2]+[1] & \text{ if $b>0$ and $a \le c$},\\
\bfm - [2] & \text{ if $a > c$}.
\end{cases}
\end{align*}
Note that 
\begin{align*}
 \ep_1(\bfm) = b + \max(c-a, 0), \quad \ep_2(\bfm) = a, \quad \eps_1(\bfm) = c, \quad \eps_2(\bfm) = b + \max(a-c, 0).
\end{align*}

\begin{prop} \label{prop: V A2}
For a dominant monomial
$m = \oprod_{k \in \Z} \left( Y_{\cdual^k(1,2)}^{a_k} Y_{\cdual^k(2,1)}^{b_k}  Y_{\cdual^k(1,0)}^{c_k} \right)$, we have 
$$
\bR_1([V(m)]) = [V(m')]\qquad \text{ and }\qquad \bR_2([V(m)]) = [V(m'')],
$$
where 
\begin{align*}
m' &= \oprod_{k \in \Z} \left( Y_{\cdual^k(1,2)}^{\min\{ a_k, c_k \}} Y_{\cdual^k(2,1)}^{b_k + \max\{ a_k-c_k, 0 \}}  Y_{\cdual^k(1,0)}^{b_{k-1} + \max\{ c_{k-1}-a_{k-1}, 0 \} } \right), \\
m'' &= \oprod_{k \in \Z} \left( Y_{\cdual^k(1,2)}^{ a_{k-1} + c_k - \min\{ a_{k-1}, b_k \}} Y_{\cdual^k(2,1)}^{\min\{ a_{k-1}, b_k \}}  Y_{\cdual^k(1,0)}^{b_{k} + c_k - \min\{ a_{k-1}, b_{k} \} } \right).
\end{align*}
\end{prop}
\begin{proof}
Let $\bfm = a[2] + b[1,2] + c[1]$. By the definition of $\tcT_i$(see \eqref{Eq: tSi}), we have 
\begin{align*}
\tcT_1(\bfm) & = \cT_1( (a+b)[2] + \min(a,c) [1] )
 = \min(a,c) [2] + (a+b-\min(a,c))[12] \\
 &= \min(a,c) [2] + (b + \max(a-c,0))[12]
\end{align*}
and 
\begin{align*}
\tcT_2(\bfm) & = \cT_2( b [1,2]+ c [1] )
 = c [2] + (b+c)[1].
\end{align*}
Then the assertion follows from the definition of $R_i$ (see \eqref{Eq: def of Ri}), \eqref{Eq: [a,b]} and Corollary \ref{Cor: R=sig}.
\end{proof}

\begin{example}
Under the identification $\KG_2 : K(\catCO)  \buildrel \sim \over \longrightarrow \biGr_{3}$, we have $[V(Y_{1,0})] = \vP_{1,3,4}$. Note that $[V(Y_{1,0})]$ is prime.

\bni
\item Lemma \ref{Lem: P A2} says that 
\begin{align*}
\sigma_2(\vP_{1,3,4}) &= 
d \left( v_{1}, d(v_2, v_4, v_5) v_3-v_2, v_4)\right) \\
&= 
d(v_2,v_4,v_5)d(v_1,v_3,v_4) - d(v_1,v_2,v_4) \\
&= \vP_{1,3,4} \vP_{2,4,5} - \vP_{1,2,4}\\
&= \vP_{1,4,5},
\end{align*}
where the last equality follows by applying 
$\bfi = (4,5) \in \cvC{2}$ and $\bfj = (1,2,3,4) \in \cvC{4}$ to the Pl\"{u}cker relation \eqref{Eq: Plucker}.
By the same manner as above, we have 
$$
\sigma_2(\vP_{2,3,5}) = \vP_{2,3,6}.
$$
Proposition \ref{prop: V A2} tells us that 
$$
\bR_2[V(Y_{1,0})] = [V(Y_{1,2}Y_{1,0})], \qquad 
\bR_2[V(Y_{2,3})] = [V(Y_{2,5}Y_{2,3})].
$$
Since $ \bR_i = \sigma_i$ by Corollary \ref{Cor: R=sig}, we have 
\begin{align} \label{Eq: ex1}
[V(Y_{1,2}Y_{1,0})] = \vP_{1,4,5}, \qquad 
[V(Y_{2,5}Y_{2,3})] = \vP_{2,3,6}.
\end{align}
Corollary \ref{Cor: real} implies that $[V(Y_{1,2}Y_{1,0})]$ and $[V(Y_{2,5}Y_{2,3})]$ are prime. Note that $[V(Y_{1,2}Y_{1,0})]$ and $[V(Y_{2,5}Y_{2,3})]$ are cluster variables.  
The equalities \eqref{Eq: ex1} can be found in the table 1 in \cite{CDFL20}.
Here is a remark about the notation $Y_{i,a}$ in \cite{CDFL20}. Due to our choice of convention (see Remark \ref{rmk: convention}), the variable $Y_{i,a}$ in the paper matches with $Y_{3-i, -a}$ used in \cite{CDFL20} for each $(i,a)$. 

\item By the same manner as above, we have 
\begin{align*}
& \sigma_1 \sigma_2(\vP_{1,3,4}) = \sigma_1(\vP_{1,4,5}) = \vP_{2,4,5}, \\
& \bR_1\bR_2[V(Y_{1,0})] = \bR_1[V(Y_{1,2}Y_{1,0})] = [V(Y_{1,2})].
\end{align*}
We then have $[V(Y_{1,2})] = \vP_{2,4,5}$. Note that $\sigma_1 \sigma_2 = \sh_1$ and 
$\bR_1\bR_2 = \shift_1$ (see Lemma \ref{Lem: Ti respects B}).

Similarly, we have 
\begin{align*}
& \sigma_2^2(\vP_{1,3,4}) = \sigma_2(\vP_{1,4,5}) = \vP_{1,4,6}, \\
& \bR_2^2[V(Y_{1,0})] = \bR_2[V(Y_{1,2}Y_{1,0})] = [V(Y_{2,5}Y_{1,2}Y_{1,0})],
\end{align*}
which tells us that 
\begin{align} \label{Eq: ex2}
[V(Y_{2,5}Y_{1,2}Y_{1,0})] = \vP_{1,4,6}.
\end{align}
The equality \eqref{Eq: ex2} also can be found in the table 1 in \cite{CDFL20}.
By Corollary \ref{Cor: real}, we know that 
 $[V(Y_{2,5}Y_{1,2}Y_{1,0})]$ is prime. 

\item Since $V(Y_{2,3})$ is the right dual of $V(Y_{1,0})$, we have the exact sequence
$$
0 \longrightarrow 1 \longrightarrow V(Y_{2,3}) \tens V(Y_{1,0})
\longrightarrow V(Y_{2,3}Y_{1,0}) \longrightarrow 0.
$$
It follows from $ [V(Y_{1,0})] = \vP_{1,3,4}$ and $ [V(Y_{2,3})] = \vP_{2,3,5}$ that
$$
[V(Y_{2,3}Y_{1,0})] = \vP_{1,3,4}\vP_{2,3,5}-1.
$$
Since $\cdual(1,0) = (2,3)$, Proposition \ref{prop: V A2} says that 
$$
\bR_2[V(Y_{2,3}Y_{1,0})] = [V(Y_{2,5}Y_{2,3}Y_{1,2}Y_{1,0})].
$$
By (i) and Lemma \ref{Lem: P A2}, we have 
\begin{align*}
\sigma_2(\vP_{1,3,4}\vP_{2,3,5}-1) = \vP_{1,4,5}\vP_{2,3,6}-1,
\end{align*}
which says that 
$$
[V(Y_{2,5}Y_{2,3}Y_{1,2}Y_{1,0})] = \vP_{1,4,5}\vP_{2,3,6}-1.
$$

Using the Pl\"{u}cker relations, we have 
\begin{align*}
[V(Y_{2,5}Y_{2,3}Y_{1,2}Y_{1,0})] &= -1+\vP_{1,4,5}\vP_{2,3,6} \\
&= -1 + \vP_{1,4,5} (-\vP_{3,5,6} + \vP_{2,3,5}\vP_{3,4,6}) \\
&= -1 + (- \vP_{1,4,5})\vP_{3,5,6} - (-\vP_{1,4,5})\vP_{2,3,5}\vP_{3,4,6}\\
&= -1 + (\vP_{1,2,4} - \vP_{1,3,4}\vP_{2,4,5} )\vP_{3,5,6} \\ 
& \quad - (\vP_{1,2,4} - \vP_{1,3,4}\vP_{2,4,5} )\vP_{2,3,5}\vP_{3,4,6}\\
&= -1 + \vP_{1,2,4}\vP_{3,5,6} - \vP_{1,3,4}\vP_{2,4,5}\vP_{3,5,6} \\
& \quad - \vP_{1,2,4}\vP_{2,3,5}\vP_{3,4,6} + 
\vP_{1,3,4}\vP_{2,3,5}\vP_{2,4,5}\vP_{3,4,6}.
\end{align*}
The last expression is given in Example 1.4 in \cite{CDFL20}.
Note that, since $[V(Y_{2,3}Y_{1,0})]$ is real by \cite[Lemma 2.28]{KKOP23}, 
$[V(Y_{2,5}Y_{2,3}Y_{1,2}Y_{1,0})]$ is real by Corollary \ref{Cor: real}.
\ee
\end{example}


\begin{thebibliography}{99}
	
\bibitem {AK97} T. Akasaka and M. Kashiwara, {\it Finite-dimensional representations of quantum affine algebras}, Publ. RIMS. Kyoto Univ., {\bf33} (1997), 839--867.


 
\bibitem {ADS14} I. Assem, G. Dupont, and R. Schiffler, \textit{On a category of cluster algebras}, J. Pure Appl. Algebra,
\textbf{218} (3), 553--582, 2014. 


\bibitem {BKM14}
J.~Brundan, A.~Kleshchev, and P.~J.~McNamara, \emph{Homological properties of finite-type Khovanov-Lauda-Rouquier algebras},  Duke Math. J. \textbf{163} (2014), no. 7, 1353--1404. 





\bibitem {CDFL20} W. Chang, B. Duan, C. Fraser, J.-R. Li, \textit{Quantum affine algebras and Grassmannians},
Math. Z. \textbf{296} (2020), no.3--4, 1539--1583.


\bibitem {CP94} V.~Chari and A.~Pressley,  \textit{A guide to quantum groups}, Cambridge University Press, Cambridge, 1994. xvi+651 pp.






\bibitem {Fra20}
C.~Fraser, {\em Braid group symmetries of Grassmannian cluster algebras},  Selecta Math. (N.S.) \textbf{26} (2020), no. 2, Paper No. 17, 51 pp. 


\bibitem {FHOO22} R. Fujita, D. Hernandez, S.-j. Oh, and H. Oya, \emph{Isomorphisms among quantum Grothendieck rings and propagation of positivity}, J. Reine Angew. Math. \textbf{785}(2022), 117--185.


\bibitem {FHOO23} R. Fujita, D. Hernandez, S.-j. Oh, and H. Oya, \textit{Isomorphisms among quantum Grothendieck rings and cluster algebras}, arXiv:2304.02562. 

\bibitem {FO21}
R.~Fujita, and S.-j.~Oh, 
\textit{$Q$-data and representation theory of untwisted quantum affine algebras} Comm. Math. Phys. \textbf{384} (2021), no.2, 1351--1407.



\bibitem {FR99}
Edward Frenkel and Nicolai Reshetikhin, \emph{The {$q$}-characters of
  representations of quantum affine algebras and deformations of
  {$\mathscr{W}$}-algebras}, Recent developments in quantum affine algebras and
  related topics ({R}aleigh, {NC}, 1998), Contemp. Math., vol. 248, Amer. Math.
  Soc., Providence, RI, 1999, pp.~163--205. \MR{1745260}

  


\bibitem {GG18} J. E. Grabowski and S. Gratz, \textit{Graded quantum cluster algebras of infinite rank as colimits}, Journal of Pure and Applied Algebra, Volume \textbf{222}, Issue 11, 2018, 3395--3413.





\bibitem {HK02} J.~Hong and S.-J.~Kang,
\emph{Introduction to Quantum Groups and Crystal Bases},
Graduate Studies in Mathematics, \textbf{42}. American Mathematical Society, Providence, RI, 2002.


\bibitem {Her03}
David Hernandez, \emph{{$t$}-analogues des op\'{e}rateurs d'\'{e}crantage
  associ\'{e}s aux {$q$}-caract\`eres}, Int. Math. Res. Not. (2003), no.~8,
  451--475. \MR{1947987}

\bibitem {Her04}
\bysame, \emph{Algebraic approach to {$q,t$}-characters}, Adv. Math.
  \textbf{187} (2004), no.~1, 1--52. \MR{2074171}
  
\bibitem {HL10} D. Hernandez and B. Leclerc, \textit{Cluster algebras and quantum affine algebras}, Duke Math. J. \textbf{154} (2) (2010), 265--341.

\bibitem {HL15} \bysame, {\em Quantum Grothendieck rings and derived Hall algebras},   J. Reine Angew. Math., {\bf 701} (2015), 77--126.


\bibitem {HO19}
D.~Hernandez and H.~Oya, \emph{Quantum Grothendieck ring isomorphisms, cluster algebras and Kazhdan-Lusztig algorithm}, Adv. Math. \textbf{347} (2019), 192--272.



\bibitem {JLO23B}
Il-Seung Jang, Kyu-Hwan Lee, and Se-jin Oh, \emph{Braid group action on quantum
	virtual grothendieck ring through constructing presentations}, Preprint,
arxiv{2305.19471}, 2023.



\bibitem {Kac} V. Kac, \newblock{\it Infinite Dimensional Lie Algebras}, \newblock{3rd ed., Cambridge University Press, Cambridge, 1990}.




\bibitem {K91}
Masaki Kashiwara, \emph{On crystal bases of the $q$-analogue of universal
  enveloping algebras}, Duke Mathematical Journal \textbf{63} (1991), no.~2,
  465--516.

\bibitem {K93}
\bysame, \emph{Global crystal bases of quantum groups}, Duke Math. J.
  \textbf{69} (1993), no.~2, 455--485. \MR{1203234}

\bibitem {K95}
\bysame, \emph{On crystal bases}, Representations of groups (Banff, AB, 1994)
  \textbf{16} (1995), 155--197.

\bibitem {Kas02} \bysame, {\it On level zero representations of quantum affine algebras}, Duke. Math. J., {\bf112} (2002), 117--175.



\bibitem {KKOP20} 
M.~Kashiwara, M.~Kim, S.-j.~Oh, and E.~Park, {\it Monoidal categorification and quantum affine algebras},  Compos. Math. \textbf{156} (2020), no.5, 1039--1077.


\bibitem {KKOP21} M. Kashiwara, M. Kim, S.-j. Oh, E. Park, \textit{Braid group action on the module category of quantum affine algebras},
Proc. Japan Acad. Ser. A Math. Sci. \textbf{97} (3), 2021, 13--18.



\bibitem {KKOP23}  \bysame,  \newblock{\em PBW theory for quantum affine algebras}, J. Eur. Math. Soc. (2023), DOI 10.4171/JEMS/1323.


\bibitem {KKOP24A}  \bysame,  \newblock{\em Braid symmetries on bosonic extensions}, arXiv:2408.07312.

\bibitem{KKOP24C}  \bysame,  \newblock{\em Global bases for Bosonic extensions of quantum unipotent coordinate rings}, arXiv:2406.13160.




\bibitem {KKOP24B}  \bysame, \newblock{\em Monoidal categorification and quantum affine algebras II}, Invent. Math.236(2024), no.2, 837--924.






\bibitem {KP22} 
M. Kashiwara and E. Park,  {\em Categorical crystals for quantum affine algebras}, Categorical crystals for quantum affine algebras, J. Reine Angew. Math. \textbf{792} (2022), 223--267.


\bibitem {LR08} V. Lakshmibai and K. N. Raghavan, \textit{Standard monomial theory: invariant theoretic approach}, Encyclopaedia of Mathematical Sciences \textbf{137}, Springer, Berlin, 2008.


\bibitem {L93}
\bysame, {\em Introduction to quantum groups}, Progr. Math. \textbf{110}  Birkh\"{a}user, 1993.

\bibitem {Nak04}
Hiraku Nakajima, \emph{Quiver varieties and {$t$}-analogs of {$q$}-characters
  of quantum affine algebras}, Ann. of Math. (2) \textbf{160} (2004), no.~3,
  1057--1097. \MR{2144973}

\bibitem {Nak04A}
H.~Nakajima, \emph{Extremal weight modules of quantum affine algebras}, Representation theory of algebraic groups and quantum groups. Adv. Stud. Pure Math., \textbf{40}, 343--369. Mathematical Society of Japan, Tokyo, 2004




\bibitem {OP24}
S.-j.~Oh and E.~Park, \emph{PBW theory for Bosonic extensions of quantum groups}, arXiv:2401.04878.


\bibitem {P23} Euiyong Park, {\em Braid group action on extended crystals}, Adv. Math. \textbf{429} (2023), Paper No. 109193, 28 pp.


\bibitem {Saito94}
Y.~Saito, \emph{PBW basis of quantized universal enveloping algebras}, Publ. Res. Inst. Math. Sci. \textbf{30} (1994), no. 2, 209--232. 

\bibitem {Sco06} J. S. Scott, \textit{Grassmannians and cluster algebras}, Proc. London Math. Soc. (3) \textbf{92} (2006), no.2, 345--380.



\bibitem {VV03}
M.~Varagnolo and E.~Vasserot, \emph{Perverse sheaves and quantum {G}rothendieck
  rings}, Studies in memory of {I}ssai {S}chur ({C}hevaleret/{R}ehovot, 2000),
  Progr. Math., vol. 210, Birkh\"{a}user Boston, Boston, MA, 2003,
  pp.~345--365. \MR{1985732}


\bibitem {V01} 
M. Vazirani, {\em Parameterizing Hecke algebra modules: Bernstein–Zelevinsky multisegments, Kleshchev multipartitions, and
crystal graphs}, Transform. Groups {\bf 7} (2002), no. 3, 267--303. 


\end{thebibliography}
\end{document}